\documentclass[reqno]{amsart}

\usepackage[usenames,dvipsnames,svgnames,table]{xcolor}
\usepackage{IEEEtrantools}

\usepackage{amssymb,latexsym,amsfonts,amsmath}
\usepackage{graphicx}
\usepackage{eso-pic}
\usepackage{epsfig}
\usepackage{graphicx}
\usepackage{mathtools}
\usepackage{tikz}
\usetikzlibrary{automata}
\usetikzlibrary{shapes}
\usetikzlibrary{calc,shapes,arrows}
\usepackage{amssymb,latexsym,amsfonts,amsmath}
\usepackage{graphicx}
\usepackage{mdwlist}
\usepackage{refcount}
\usepackage{comment}

\newtheorem{theorem}{Theorem}[section]
\newtheorem{lemma}[theorem]{Lemma}

\newtheorem{definition}[theorem]{Definition}

\newtheorem{remark}[theorem]{Remark}
\numberwithin{equation}{section}
\newtheorem{assumption}{Assumption}

\xdefinecolor{MATLABblue}{rgb}{0.0,0.45,.74}
\xdefinecolor{MATLABred}{rgb}{0.85,0.33,.1}

\newcommand{\R}{{\mathbb{R}}}

\newcommand{\N}{{\mathbb{N}}}

\newcommand{\intcc}[1]{\ensuremath{{\left[#1\right]}}}
\newcommand{\intoc}[1]{\ensuremath{{\left]#1\right]}}}
\newcommand{\intco}[1]{\ensuremath{{\left[#1\right[}}}
\newcommand{\intoo}[1]{\ensuremath{{\left]#1\right[}}}
\newcommand{\as}{\overset{\ra}{s}}
\newcommand{\st}{{\rm s.t.}}

\usepackage[normalem]{ulem}

\usepackage{graphicx,keyval,trig}
\usepackage{amsmath,amssymb,amsfonts,amssymb,dsfont}
\usepackage{mathtools, mathrsfs}

\DeclareMathOperator{\im}{im}
\DeclareMathOperator{\ke}{ker}

\newcommand{\ra}{\rightarrow}

\newcommand{\ol}{\overline}
\newcommand{\Let}{:=}

\begin{document}

\begin{abstract}
In this paper we propose a compositional scheme for the construction of abstractions for networks of control systems using the interconnection matrix and joint dissipativity-type properties of subsystems and their abstractions. In the proposed framework, the abstraction, itself a control system (possibly with a lower dimension), can be used as a substitution of the original
system in the controller design process. Moreover, we provide a procedure for constructing abstractions of a class of nonlinear control systems by using the bounds on the slope of system nonlinearities. We illustrate the proposed results on a network of linear control systems by constructing its abstraction in a compositional way without requiring any condition on the number or gains of the subsystems. We use the abstraction as a substitute to
synthesize a controller enforcing a certain linear temporal logic
specification. This example particularly elucidates the effectiveness of dissipativity-type compositional reasoning for large-scale systems.
\end{abstract}

\title[Compositional abstraction for networks of control systems: A dissipativity approach]{Compositional abstraction for networks of control systems: A dissipativity approach}

\author[M. Zamani]{Majid Zamani$^1$} 
\author[M. Arcak]{Murat Arcak$^2$} 
\address{$^1$Department of Electrical and Computer Engineering, Technical University of Munich, D-80290 Munich, Germany.}
\email{zamani@tum.de}
\urladdr{http://www.hcs.ei.tum.de}
\address{$^2$Department of Electrical Engineering and Computer Sciences, University of California, Berkeley, CA, USA.}
\email{arcak@berkeley.edu}
\urladdr{http://www.eecs.berkeley.edu/$\sim$arcak}

\maketitle

\section{Introduction}
Modern applications, e.g. power networks, biological networks, internet congestion control, and manufacturing systems,
are large-scale networked systems and inherently difficult to analyze and control. Rather than tackling the network as a whole, an approach that severely restricts the capability of existing techniques to deal with many numbers of subsystems, one can develop compositional schemes that provide network-level certifications from main structural properties of the subsystems and their interconnections. 

In the past few years, there have been several results on the compositional abstractions of control systems. Early results include compositional abstractions of control systems \cite{TabuadaPappasLima04,Frehse05,KvdS10} which are useful for verification rather than synthesis. Those results employ exact notions of abstractions based on simulation relations \cite{Frehse05,KvdS10}
and simulation maps \cite{TabuadaPappasLima04}, for which constructive methodologies
exist only for rather restricted classes of control systems. In contrast to the exact notions, the compositional approximate abstractions were introduced recently which are useful for the controller synthesis. Examples include compositional construction of \emph{finite} abstractions of linear and nonlinear control systems \cite{tazaki1,PPdB14} and of \emph{infinite} abstractions of nonlinear control systems \cite{majid17,majid20} and a class of stochastic hybrid systems \cite{majid19}. In those works, the abstraction (finite or infinite with possibly a lower dimension) can be used as a substitution of the original
system in the controller design process. The proposed results in \cite{tazaki1,PPdB14,majid17,majid20,majid19} use the small-gain type conditions to facilitate the compositional construction of abstractions. The resulting small-gain type requirements intrinsically condition the spectral radius of the interconnection matrix which, in general, depends on the size of the graph and can be violated or deteriorated as the number of subsystems grows \cite{das}.

In this work we propose a novel compositional framework for the construction of infinite abstractions of networks of control systems using dissipativity theory. 
%Here, the abstraction is itself a control system (possibly with a lower dimension) and can be used as a substitution of the concrete control
%system in the controller design process. 
First, we adapt the notion of storage function from dissipativity theory \cite{murat16} to quantify the joint dissipativity-type properties of control subsystems and their abstractions. Given a network of control subsystems and their storage functions, we propose conditions based on the interconnection matrix and joint dissipativity-type properties of subsystems and their abstractions guaranteeing that the network of abstractions quantitatively approximate the behaviours of the network of concrete subsystems. The proposed compositionality conditions can enjoy specific interconnection structures and provide scale-free compositional abstractions for large-scale control systems without requiring any condition on the number or gains of the subsystems; we illustrate this point with an example in Section \ref{case}. Furthermore, we provide a geometric approach on the construction of abstractions for a class of nonlinear control systems and of their corresponding storage functions by using the bounds on the slope of system nonlinearities. 
%Finally, we illustrate the effectiveness of the proposed results by providing abstractions of a class of networks of linear control systems in which the compositionality conditions are automatically satisfied independently of the number or gains of the subsystems.

{\bf Related Work.} 
Compositional construction of infinite abstractions of networks of control systems is also proposed in \cite{majid17,majid20}. While in \cite{majid17,majid20} small-gain type conditions are used to facilitate the compositional construction of abstractions, here we use dissipativity-type conditions. The small-gain type requirements inherently condition the spectral radius of the interconnection matrix which, in general, depends on the size of the graph and can be dissatisfied as the number of subsystems grows \cite{das}. On the other hand, this is not necessarily the case with broader dissipativity-type conditions and in fact the compositionality requirements may not condition the number or gains of the subsystems at all when the interconnection matrix enjoys some properties (cf. Section \ref{case}). Although the results in \cite{majid17,majid20} provide constructive procedures to
determine abstractions of linear control systems, we propose techniques on the construction of abstractions for a class of nonlinear control systems by using the bounds on the slope of systems nonlinearities. The results in \cite{majid17,majid20} assume that the internal input and output space dimensions of each component in a network are equal to the corresponding ones of its abstraction which is not the case in this paper. While the interconnection matrix in \cite{majid17,majid20} is a permutation one, the one in this paper can be any general interconnection matrix.
%Apart from these two fundamental differences, there exist extra key benefits provided by the proposed results here which will be elaborated in details in the paper.

%The results in \cite{murat16,meissen} study only compositional stability, performance, and stability certification using dissipativity. This paper advances the ideas proposed in those works and proposes compositional construction of abstractions for networks of control systems. 
The recent results in \cite{murat16,meissen} establish only stability or stabilizability of networks of control systems compositionally using dissipativity properties of components. On the other hand, the results here provide construction of abstractions of networks of control systems compositionally using abstractions of components and their joint dissipativity-type properties.

\section{Control Systems}
\subsection{Notation}
The sets of nonnegative integer and real numbers are denoted by $\N$ and $\R$, respectively. Those symbols are footnoted with subscripts to restrict them in
the usual way, e.g. $\R_{>0}$ denotes the positive real numbers. The symbol $ \R^{n\times m} $ denotes the vector space of real matrices with $ n $ rows and $ m $ columns. The symbols $\mathbf{1}_{n}$, $0_n$, $I_n$, and $0_{n\times{m}}$ denote the vector in $\R^n$ with all its elements to be one, the zero vector, identity and zero matrices in $\R^n$, $\R^{n\times n}$, and $\R^{n\times{m}}$, respectively. For $a,b\in\R$ with $a\le b$, the closed, open, and half-open intervals in $\R$ are denoted by $\intcc{a,b}$,
$\intoo{a,b}$, $\intco{a,b}$, and $\intoc{a,b}$, respectively. For $a,b\in\N$ and $a\le b$, the symbols $\intcc{a;b}$, $\intoo{a;b}$, $\intco{a;b}$, and $\intoc{a;b}$
denote the corresponding intervals in $\N$.
Given $N\in\N_{\geq1}$, vectors $x_i\in\R^{n_i}$, $n_i\in\N_{\geq1}$ and $i\in\intcc{1;N}$, we
use $x=[x_1;\ldots;x_N]$ to denote the concatenated vector in $\R^n$ with
$n=\sum_{i=1}^N n_i$. Given a vector \mbox{$x\in\mathbb{R}^{n}$}, $\Vert x\Vert$ denotes the Euclidean norm of $x$. 
Note that given any $x\in\R^N$, $x\geq0$ iff $x_i\geq0$ for any $i\in[1;N]$. 
Given a symmetric matrix $A$, $\lambda_{\max}(A)$ and $\lambda_{\min}(A)$ denote maximum and minimum eigenvalues of $A$. We denote by $\mathsf{diag}(M_1,\ldots,M_N)$ the block diagonal matrix with diagonal matrix entries $M_1,\ldots,M_N$.  

Given a function $f:\R^n\rightarrow \R^m$ and $0_m\in\R^m$, we simply use
$f\equiv 0$ to denote that $f(x)=0_m$ for all $x\in\R^n$. Given a function \mbox{$f:\mathbb{R}_{\geq0}\rightarrow\mathbb{R}^n$}, the (essential) supremum of $f$ is denoted by $\Vert f\Vert_{\infty} \Let \text{(ess)sup}\{\Vert f(t)\Vert,t\geq0\}$. A continuous function \mbox{$\gamma:\mathbb{R}_{\geq0}\rightarrow\mathbb{R}_{\geq0}$}, is said to belong to class $\mathcal{K}$ if it is strictly increasing and \mbox{$\gamma(0)=0$}; $\gamma$ is said to belong to class $\mathcal{K}_{\infty}$ if \mbox{$\gamma\in\mathcal{K}$} and $\gamma(r)\rightarrow\infty$ as $r\rightarrow\infty$. A continuous function $\beta:\mathbb{R}_{\geq0}\times\mathbb{R}_{\geq0}\rightarrow\mathbb{R}_{\geq0}$ is said to belong to class $\mathcal{KL}$ if, for each fixed $t$, the map $\beta(r,t)$ belongs to class $\mathcal{K}$ with respect to $r$ and, for each fixed nonzero $r$, the map $\beta(r,t)$ is decreasing with respect to $t$ and $\beta(r,t)\rightarrow 0$ as \mbox{$t\rightarrow\infty$}.

\subsection{Control systems}\label{control_system}
The class of control systems studied in this paper is formalized in
the following definition.
\begin{definition}
\label{Def_control_sys}A \textit{control system} $\Sigma$ is a tuple
$\Sigma=(\mathbb{R}^{n},\R^m,\R^p,\mathcal{U},\mathcal{W},f,\R^{q_1},\R^{q_2},h_1,h_2)$, where $\R^n$, $\R^m$, $\R^p$, $\R^{q_1}$, and $\R^{q_2}$ are the state, external input, internal input, external output, and internal
output spaces, respectively, and
\begin{itemize}
\item $\mathcal{U}$ and $\mathcal{W}$ are subsets of the sets of all measurable functions of time, from open intervals in $\mathbb{R}$ to $\R^m$ and $\R^p$, respectively; 
\item \mbox{$f:\mathbb{R}^{n}\times \R^m\times\R^p\rightarrow\mathbb{R}^{n}$} is a continuous map
satisfying the following Lipschitz assumption: for every compact set
\mbox{$\mathsf{D}\subset\mathbb{R}^{n}$}, there exists a constant $Z\in\mathbb{R}_{>0}$ such that $\Vert
f(x,u,w)-f(y,u,w)\Vert\leq Z\Vert x-y\Vert$ for all $x,y\in \mathsf{D}$, all $u\in\R^m$, and all $w\in\R^p$;
\item $h_1:\R^n\rightarrow\R^{q_1}$ is the external output map;
\item $h_2:\R^n\rightarrow\R^{q_2}$ is the internal output map.
\end{itemize}
\end{definition}

A locally absolutely continuous curve \mbox{$\xi:]a,b[\rightarrow\mathbb{R}^{n}$} is a
\textit{state trajectory} of $\Sigma$ if there exist input trajectories $\upsilon\in\mathcal{U}$ and $\omega\in\mathcal{W}$
satisfying:
\begin{IEEEeqnarray}{c}\label{e:ode}\Sigma:\left\{
  \begin{IEEEeqnarraybox}[\relax][c]{rCl}
    \dot\xi(t)&=&f(\xi(t),\upsilon(t),\omega(t)),\\
    \zeta_1(t)&=&h_1(\xi(t)),\\
    \zeta_2(t)&=&h_2(\xi(t)),
  \end{IEEEeqnarraybox}\right.
\end{IEEEeqnarray}
for almost all $t\in$ $]a,b[$. We call the tuple $(\xi,\zeta_1,\zeta_2,\upsilon,\omega)$ a \emph{trajectory}
of $\Sigma$, consisting of a state trajectory $\xi$, output trajectories $\zeta_1$ and $\zeta_2$, and input trajectories
$\upsilon$ and $\omega$, that satisfies \eqref{e:ode}. We also denote by $\xi_{x\upsilon\omega}(t)$ the state reached at time $t$
under the inputs $\upsilon\in\mathcal{U},\omega\in\mathcal{W}$ from the initial condition $x=\xi_{x\upsilon\omega}(0)$; the state $\xi_{x\upsilon\omega}(t)$ is
uniquely determined due to the assumptions on $f$ \cite{sontag1}. We also denote by $\zeta_{1_{x\upsilon\omega}}(t)$ and $\zeta_{2_{x\upsilon\omega}}(t)$ the corresponding external and internal output value of $\xi_{x\upsilon\omega}(t)$, respectively, i.e. $\zeta_{1_{x\upsilon\omega}}(t)=h_1(\xi_{x\upsilon\omega}(t))$ and $\zeta_{2_{x\upsilon\omega}}(t)=h_2(\xi_{x\upsilon\omega}(t))$.

We call $\zeta_1$ an \emph{external output
trajectory}, $\zeta_2$ an \emph{internal output trajectory}, $\upsilon$ an \emph{external input
trajectory}, and $\omega$ an \emph{internal input trajectory} mainly because $\zeta_2$ and $\omega$ are used only for the interconnection purposes and $\zeta_1$ and $\upsilon$ remain available after any interconnection; see Definition \ref{d:ics} later for more detailed information.
%A control system $\Sigma$ is said to be forward complete if every state trajectory is defined on an interval of the form $]a,\infty[$. 
%Standard sufficient and necessary conditions for a control system to be forward complete can be found in \cite{sontag}. Throughout this paper, we assume every control system $\Sigma$ is forward complete. 

\begin{remark}
If the control system $\Sigma$ does not have internal inputs and outputs, the definition of control systems in Definition \ref{Def_control_sys} reduces to tuple $\Sigma=(\mathbb{R}^{n},\R^m,\mathcal{U},f,\R^{q},h)$ and the map $f$ becomes $f:\mathbb{R}^{n}\times\R^m\rightarrow\mathbb{R}^{n}$. Correspondingly, equation \eqref{e:ode} describing the evolution of system trajectories reduces to:
\begin{IEEEeqnarray*}{c}\Sigma:\left\{
  \begin{IEEEeqnarraybox}[\relax][c]{rCl}
    \dot\xi(t)&=&f(\xi(t),\upsilon(t)),\\
    \zeta(t)&=&h(\xi(t)).  \end{IEEEeqnarraybox}\right.
\end{IEEEeqnarray*}
\end{remark}

\section{Storage and Simulation Functions}
First, we introduce a notion of so-called storage functions, adapted from the notion of storage functions from dissipativity theory \cite{willems2,murat16}. While the notion of storage functions in \cite{willems2,murat16} characterizes the correlation of inputs and outputs of a single control system, the proposed notion of storage functions here characterizes the joint correlation of inputs and outputs of two different control systems. In the case that two control systems are the same and have only internal inputs and outputs, our notion of storage functions recovers the one of incremental storage functions introduced in \cite{stan}.

\begin{definition}\label{d:stf} Let
$\Sigma=(\R^n,\R^m,\R^p,\mathcal{U},\mathcal{W},f,\R^{q_1},\R^{q_2},h_1,h_2)$ and
$\hat\Sigma=(\R^{\hat n},\R^{\hat m},\R^{\hat p},\hat{\mathcal{U}},\hat{\mathcal{W}},\hat f,\R^{q_1},\R^{\hat q_2},\hat h_1,\hat h_2)$ be two control systems with the same external output space
dimension. A continuously differentiable 
function $V:\R^{n}\times \R^{\hat n}\to\R_{\ge0}$ is
called a \emph{storage function} from $\hat\Sigma$ to $\Sigma$ if there exist
$\alpha,\eta\in\mathcal{K}_\infty$, $\rho_{\mathrm{ext}}\in\mathcal{K}_\infty\cup\{0\}$, some matrices $W,\hat W,H$ of appropriate dimensions, and some symmetric matrix $X$ of appropriate dimension with conformal block partitions $X^{ij}$, $i,j\in[1;2]$, where $X^{22}\preceq0$, such that
for any $x\in\R^n$ and $\hat x\in\R^{\hat n}$ one has
\begin{align}\label{e:sf:1}
\alpha(\Vert h_1(x)-\hat h_1(\hat x)\Vert)&\le V(x,\hat x),
%\\\label{e:sf:3}\alpha_2(\Vert h_2(x)-H\hat h_2(\hat x)\Vert)&\le V(x,\hat x)
\end{align}
and $\forall \hat u\in\R^{\hat m}$ $\exists u\in\R^m$ such that $\forall\hat w\in\R^{\hat p}$ $\forall w\in\R^p$ one obtains
\begin{align}\label{inequality1}
	&\nabla V\left(x, \hat{x}\right)^T\begin{bmatrix}
f(x,u,w)\\
\hat f(\hat x,\hat u,\hat w)
\end{bmatrix} \leq-\eta (V\left(x,\hat{x}\right))+\rho_{\mathrm{ext}}(\Vert\hat
u\Vert)+\begin{bmatrix}
Ww-\hat W\hat w\\
h_2(x)-H\hat h_2(\hat x)
\end{bmatrix}^T\overbrace{\begin{bmatrix}
X^{11}&X^{12}\\
X^{21}&X^{22}
\end{bmatrix}}^{X:=}\begin{bmatrix}
Ww-\hat W\hat w\\
h_2(x)-H\hat h_2(\hat x)
\end{bmatrix}.
\end{align} 
%and for any $r\in\R_{\geq0}$ one has\footnote{Here, $\alpha_2^{-2}(r):=(\alpha_2^{-1}(r))^2$, for any $r\in\R_{\geq0}$.}
%\begin{IEEEeqnarray}{c}\label{e:sf:2}
%-\eta(r)+\lambda_{\max}(X^{22}+\epsilon\Vert X^{12}\Vert I_q)\alpha_2^{-2}(r)\leq-\varepsilon(r),
%\end{IEEEeqnarray}
\end{definition}

We use notation $\hat\Sigma\preceq\Sigma$ if
there exists a storage function $V$ from $\hat\Sigma$ to $\Sigma$. Control system $\hat\Sigma$ (possibly with $\hat n<n$) is called an abstraction of $\Sigma$. There are several key differences between the notion of storage function here and the corresponding one of simulation function in \cite[Definition 2]{majid20}. Definition 2 in \cite{majid20} requires internal signals $w,\hat w$ and $h_2(x),\hat h_2(\hat x)$ to live in the same spaces, respectively, which is not necessarily the case here. Moreover, the choice of input $u$ here satisfying \eqref{inequality1} only depends on $x$, $\hat x$, and $\hat u$, whereas in \cite[Definition 2]{majid20} it also depends on internal input $\hat w$. Finally, we should point out that if in \cite[Definition 2]{majid20} $\mu(s):=s^TPs$, for any $s\in\R^p_{\geq0}$ and some positive definite matrix $P$, then the simulation function in \cite[Definition 2]{majid20} is also a storage function as in Definition \ref{d:stf} with $W=\hat W=I_{p}$, $X^{11}=P$, and the rest of conformal block partitions of $X$ are zero.

Now, we recall the notion of simulation functions introduced 
in~\cite{GP09} with some modifications.

\begin{definition}\label{d:sf} Let
$\Sigma=(\R^n,\R^m,\mathcal{U},f,\R^{q},h)$ and
$\hat\Sigma=(\R^{\hat n},\R^{\hat m},\hat{\mathcal{U}},\hat f,\R^{q},\hat h)$ be
two control systems. A continuously differentiable 
function $V:\R^{n}\times \R^{\hat n}\to\R_{\ge0}$ is
called a \emph{simulation function} from $\hat\Sigma$ to $\Sigma$ if there exist
$\alpha,\eta\in\mathcal{K}_\infty$ and $\rho_{\mathrm{ext}}\in\mathcal{K}_\infty\cup\{0\}$ such that
for any $x\in\R^n$ and $\hat x\in\R^{\hat n}$ one has
\begin{align}\label{e:sf:10}
\alpha(\Vert h(x)-\hat h(\hat x)\Vert)&\le V(x,\hat x),
%\\\label{e:sf:3}\alpha_2(\Vert h_2(x)-H\hat h_2(\hat x)\Vert)&\le V(x,\hat x)
\end{align}
and $\forall \hat u\in\R^{\hat m}$ $\exists u\in\R^m$ such that
\begin{align}\label{inequality10} 
	\nabla V\left(x, \hat{x}\right)^T\begin{bmatrix}
f(x,u)\\
\hat f(\hat x,\hat u)
\end{bmatrix} \leq-\eta (V\left(x,\hat{x}\right))+\rho_{\mathrm{ext}}(\Vert\hat
u\Vert).
\end{align} 
%and for any $r\in\R_{\geq0}$ one has\footnote{Here, $\alpha_2^{-2}(r):=(\alpha_2^{-1}(r))^2$, for any $r\in\R_{\geq0}$.}
%\begin{IEEEeqnarray}{c}\label{e:sf:2}
%-\eta(r)+\lambda_{\max}(X^{22}+\epsilon\Vert X^{12}\Vert I_q)\alpha_2^{-2}(r)\leq-\varepsilon(r),
%\end{IEEEeqnarray}
\end{definition}

%We say that a control system $\hat\Sigma$ is \emph{approximately simulated} by a control
%system $\Sigma$ or $\Sigma$ \emph{approximately simulates}
%$\hat\Sigma$, denoted by $\hat\Sigma\preceq_{\mathcal{AS}}\Sigma$, if
%there exists a simulation function $V$ from $\hat\Sigma$ to
%$\Sigma$. 
We use notation $\hat\Sigma\preceq_{\mathcal{S}}\Sigma$ if there exists a simulation function $V$ from $\hat\Sigma$ to
$\Sigma$. 
%We call $\hat \Sigma$ an
%\emph{abstraction} of $\Sigma$. 
%We say that a control system $\hat\Sigma$ is \emph{approximately bisimilar} to a control
%system $\Sigma$, denoted by $\hat\Sigma\cong_{\mathcal{AS}}\Sigma$, if
%there exists a simulation function $V$ from $\hat\Sigma$ to
%$\Sigma$ and there exists a simulation function $\widetilde V$ from $\Sigma$ to
%$\hat\Sigma$.

%\begin{remark}
%Note that we do not necessarily require a function $V$ to be a simulation function from $\hat\Sigma$ to $\Sigma$ and from $\Sigma$ to $\hat\Sigma$ in order to have $\hat\Sigma\cong_{\mathcal{AS}}\Sigma$. The next result shows that given a simulation function $V$ from $\hat\Sigma$ to $\Sigma$, one can always quantify the error between their outputs. Therefore, in the case of $\hat\Sigma\cong_{\mathcal{AS}}\Sigma$, two different simulation functions simply imply two different error bounds between the outputs of $\Sigma$ and $\hat\Sigma$ in which one can always take the maximum of them as the unified error bound between those outputs.
%\end{remark}

Let us point out the differences between Definition \ref{d:sf} here and \cite[Definition 1]{GP09}. Here, for the sake of brevity, we simply assume that for every $x$,
$\hat x$, $\hat u$, there exists $u$ so that \eqref{inequality10} holds. Whereas
in \cite[Definition 1]{GP09} the authors use an \emph{interface function} $k:\R^n\times\R^{\hat n}\times\R^{\hat m}\to\R^m$ to feed
the input $u = k(x,\hat x,\hat u)$ enforcing \eqref{inequality10}. Function $\alpha$ in \cite[Definition 1]{GP09} is assumed to be the identity. Furthermore, we frame the decay condition \eqref{inequality10} in so-called ``dissipative" form, while in \cite[Definition 1]{GP09} the decay condition is given in so-called ``implication" form.

Note that the notions of storage functions in Definition \ref{d:stf} and simulation functions in Definition \ref{d:sf} are not comparable in general. The former is defined for control systems with internal inputs and outputs while the latter is defined only for control systems without internal inputs and outputs. One can readily verify that both notions coincide for control systems without internal inputs and outputs.

%One can readily verify from Definitions \ref{d:stf} and \ref{d:sf} that every simulation function $V$ from $\Sigma=(\R^n,\R^m,\mathcal{U},f,\R^{q},h)$ to
%$\hat\Sigma=(\R^{\hat n},\R^{\hat m},\hat{\mathcal{U}},\hat f,\R^{q},\hat h)$ is also a storage function from $\Sigma$ to $\hat\Sigma$. 
The next theorem shows the importance of the existence of a simulation function by quantifying the error between the output behaviours of $\Sigma$ and the ones of its abstraction $\hat \Sigma$.
\begin{theorem}\label{theorem1}
Let 
$\Sigma=(\R^n,\R^m,\mathcal{U},f,\R^{q},h)$ and
$\hat\Sigma=(\R^{\hat n},\R^{\hat m},\hat{\mathcal{U}},\hat f,\R^{q},\hat h)$.
Suppose $V$ is a simulation function from $\hat\Sigma$ to $\Sigma$. Then, there
exist a $\mathcal{KL}$ function $\vartheta$ such that 
for any
$\hat\upsilon\in\hat{\mathcal{U}}$, $x\in\R^n$, and $\hat x\in\R^{\hat n}$, there exists $\upsilon\in{\mathcal{U}}$ such that the following inequality holds for any $t\in\R_{\ge0}$:
%\begin{IEEEeqnarray}{rCl}\label{inequality}
%\begin{IEEEeqnarraybox}[][c]{l}
\begin{align}
\label{inequality0} 
	\Vert{\zeta}_{x\upsilon}(t) -\hat\zeta_{\hat x \hat \upsilon}(t)\Vert \le& \alpha^{-1}(2\vartheta\left(V(x,\hat x),t\right ) )+\alpha^{-1}(2\eta^{-1}(2\rho_{\mathrm{ext}}(\Vert\hat \upsilon\Vert_\infty))).
%\le\\\notag\max\left\{\alpha^{-1}(V(x,\hat x)),\alpha^{-1}(\eta^{-1}\left(\rho(||\hat \nu||_\infty\right)))\right\},\label{inequality}
%\end{IEEEeqnarraybox}
%\end{IEEEeqnarray}
\end{align} 
%where $\gamma_{\mathrm{ext}}(r)=\alpha^{-1}(2\eta^{-1}(4\rho_{\mathrm{ext}}(r)))$ for any $r\in\R_{\geq0}$.
\end{theorem}

The proof of Theorem \ref{theorem1} is similar to the one of Theorem 3.5 in \cite{majid19} and is omitted due to lack of space.

%\begin{proof}
%Since $V$ is a simulation function from $\hat\Sigma$ to $\Sigma$, for any $x\in\R^n$, $\hat x\in\R^{\hat n}$, $\hat u\in\R^{\hat m}$ there exists $u\in\R^m$ such that one has
%\begin{align}\label{deriv} 
%	\nabla V\left(x, \hat{x}\right)^T\begin{bmatrix}
%f(x,u)\\
%\hat f(\hat x,\hat u)
%\end{bmatrix} \leq&-\eta (V\left(x,\hat{x}\right))+\rho_{\mathrm{ext}}(\Vert\hat
%u\Vert).
%\end{align} 
%Using the results in Lemma 4.4 in \cite{LSW96} or Lemma 3.6 in \cite{majid19}, inequality \eqref{deriv} implies the existence of a $\mathcal{KL}$ function $\vartheta$ such that for any $x\in\R^n$, $\hat x\in\R^{\hat n}$, $\hat \upsilon\in\hat{\mathcal{U}}$ there exists $\upsilon\in\mathcal{U}$ satisfying
%\begin{align}\label{boundonV}
%V(\xi_{x\upsilon}(t),\hat\xi_{\hat x\hat\upsilon}(t))\leq\vartheta(V(x,\hat x),t)+\eta^{-1}(2\rho_{\mathrm{ext}}(\Vert\hat
%\upsilon\Vert_\infty).
%\end{align}
%Using inequality \eqref{e:sf:10}, we obtain
%\begin{align}\notag
%\alpha(\Vert\zeta_{{x\upsilon}}(t)-\hat\zeta_{{\hat x\hat\upsilon}}(t)\Vert)\leq\vartheta(V(x,\hat x),t)+\eta^{-1}(2\rho_{\mathrm{ext}}(\Vert\hat
%\upsilon\Vert_\infty).
%\end{align}
%Since $\alpha\in\mathcal{K}_\infty$, one gets
%\begin{align}\notag
%&\Vert\zeta_{{x\upsilon}}(t)-\hat\zeta_{{\hat x\hat\upsilon}}(t)\Vert\leq\alpha^{-1}\big(\vartheta(V(x,\hat x),t)+\eta^{-1}(2\rho_{\mathrm{ext}}(\Vert\hat
%\upsilon\Vert_\infty))\big)\\\notag&\qquad\leq
%\alpha^{-1}(2\vartheta(V(x,\hat x),t))+\alpha^{-1}(2\eta^{-1}(2\rho_{\mathrm{ext}}(\Vert\hat
%\upsilon\Vert_\infty))),
%\end{align}
%which completes the proof.
%\end{proof}

Let us illustrate the importance of the existence of a simulation function, correspondingly inequality \eqref{inequality0}, on a simple example. Assume we are given a control system $\Sigma=(\R^n,\R^m,\mathcal{U},f,\R^{q},h)$ and interested in computing a control input $\upsilon$ to keep the output $\zeta_{x\upsilon}$ always inside a safe set $\mathsf{D}\subset\R^q$. Instead, one can compute a control input $\hat\upsilon$ for the abstraction $\hat\Sigma$ keeping the output $\hat\zeta_{\hat x\hat\upsilon}$ always inside $\mathsf{D}$ which is potentially easier due to a lower dimension of $\hat\Sigma$. The existence of a simulation function from $\hat\Sigma$ to $\Sigma$ and, hence, the inequality \eqref{inequality0} imply that there exists control input $\upsilon$ such that $\zeta_{x\upsilon}$ is always inside $\mathsf{D}^{\varepsilon}$, where $\varepsilon=\alpha^{-1}(2\vartheta\left(V(x,\hat x),0\right ) )+\alpha^{-1}(2\eta^{-1}(2\rho_{\mathrm{ext}}(\Vert\hat \upsilon\Vert_\infty)))$ and $\mathsf{D}^{\varepsilon}=\{y\in\R^p\,\,|\,\,\inf_{y'\in\mathsf{D}}\Vert y-y'\Vert\leq\varepsilon\}$. Note that one can choose initial conditions $x\in\R^n$ and $\hat x\in\R^{\hat n}$ to minimize the first term in $\varepsilon$ and, hence, to have a smaller error in the satisfaction of the desired property. 

\begin{remark}
Note that if $\alpha^{-1}$ and $\eta^{-1}$ satisfy the triangle inequality (i.e., $\alpha^{-1}(a+b)\leq\alpha^{-1}(a)+\alpha^{-1}(b)$ and $\eta^{-1}(a+b)\leq\eta^{-1}(a)+\eta^{-1}(b)$ for all $ a,b \in \R_{\geq0}$), one can divide all the coefficients $2$, appearing in the right hand side of \eqref{inequality0}, by factor $2$ to get a less conservative upper bound.
\end{remark}

\begin{remark}
Note that if one is given an interface function $k:\R^n\times\R^{\hat n}\times\R^{\hat m}\to\R^{m}$ that maps every $x$, $\hat x$, $\hat u$ to an input $u=k(x,\hat x,\hat u)$ such that \eqref{inequality10} is satisfied (similar to \cite[Definition 1]{GP09}), then input $\upsilon$ realizing \eqref{inequality0} is readily given by $\upsilon=k(\xi,\hat \xi,\hat\upsilon)$. In Section \ref{s:lin} we show how the map $k$ can be constructed for a class of nonlinear control systems.
\end{remark}

\section{Compositionality Result}
\label{s:inter}
In this section, we analyze networks of control systems and show
how to construct their abstractions together with the corresponding simulation functions by using storage functions for the subsystems. The
definition of the network of control systems is based on the notion of
interconnected systems described in~\cite{murat16}. 

\subsection{Interconnected control systems}

Here, we define the \emph{interconnected control system} as the following.

\begin{definition}\label{d:ics}
Consider $N\in\N_{\geq1}$ control subsystems $
\Sigma_i=\left(\R^{n_i},\R^{m_i},\R^{p_i},\mathcal{U}_i,\mathcal{W}_i,f_i,\R^{q_{1i}},\R^{q_{2i}},h_{1i},h_{2i}\right)$,
$i\in\intcc{1;N}$, and a static matrix $M$ of an appropriate dimension defining the coupling of these subsystems. The \emph{interconnected control
system} $\Sigma=\left(\R^{n},\R^{m},\mathcal{U},f,\R^{q},h\right)$, denoted by
$\mathcal{I}(\Sigma_1,\ldots,\Sigma_N)$, follows by $n=\sum_{i=1}^Nn_i$,
$m=\sum_{i=1}^Nm_{i}$, $q=\sum_{i=1}^Nq_{1i}$, and functions
\begin{align*}
f(x,u)&\Let\intcc{f_1(x_1,u_1,w_1);\ldots;f_N(x_N,u_N,w_N)},\\
h(x)&\Let \intcc{h_{11}(x_1);\ldots;h_{1N}(x_N)},
\end{align*}
where $u=\intcc{u_{1};\ldots;u_{N}}$, $x=\intcc{x_{1};\ldots;x_{N}}$ and
with the internal variables constrained by
$$\intcc{w_{1};\ldots;w_{N}}=M\intcc{h_{21}(x_1);\ldots;h_{2N}(x_N)}.$$
\end{definition}

An interconnection of $N$ control subsystems $\Sigma_i$ is illustrated schematically in Figure \ref{system1}.
\begin{figure}[ht]
\begin{center}
\includegraphics[width=5cm]{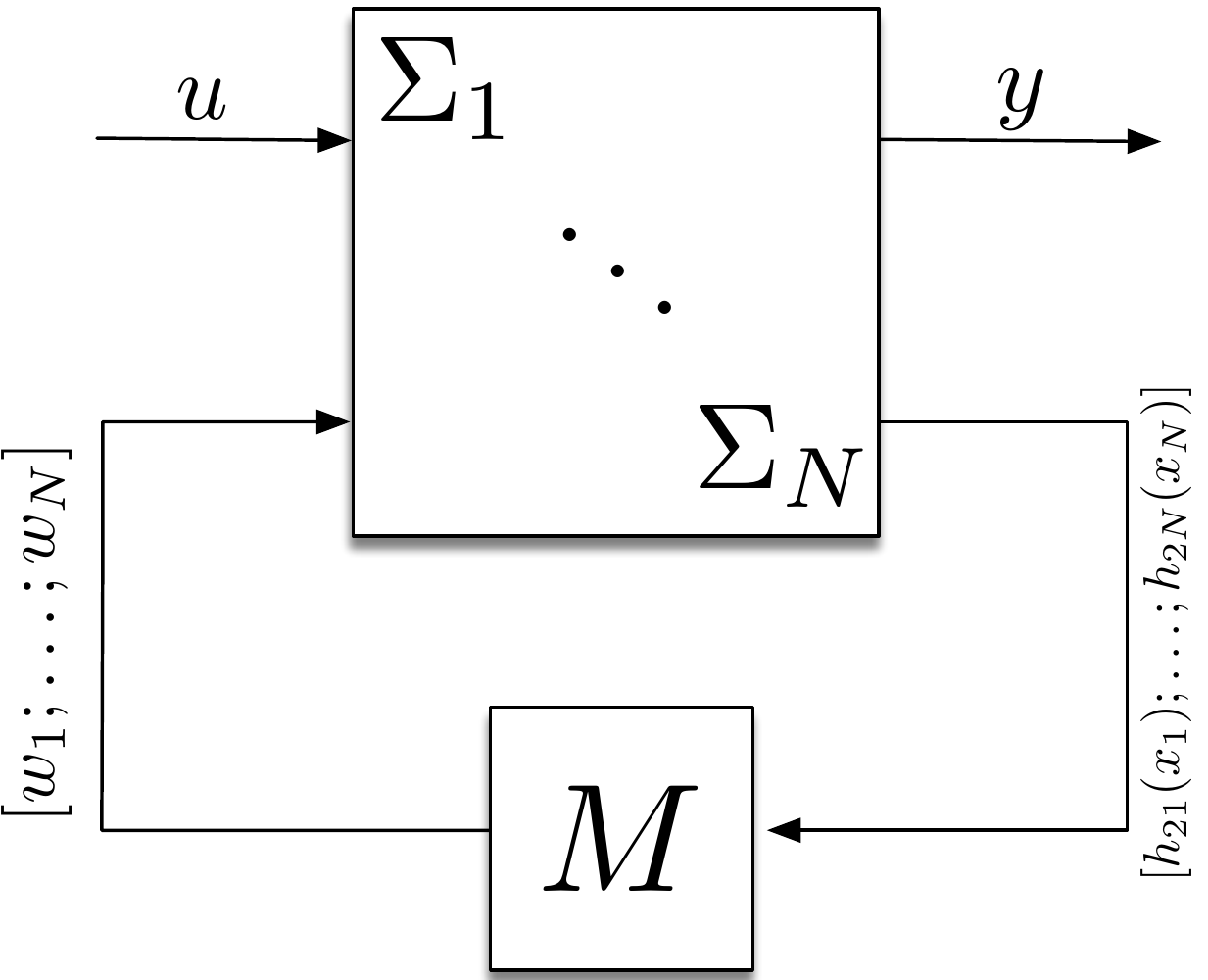}
\caption{An interconnection of $N$ control subsystems $\Sigma_1,\ldots,\Sigma_N$.}
\label{system1}
\end{center}
\end{figure}

\subsection{Composing simulation functions from storage functions}
We assume that we are given $N$ control subsystems
$\Sigma_i=\left(\R^{n_i},\R^{m_i},\R^{p_i},\mathcal{U}_i,\mathcal{W}_i,f_i,\R^{q_{1i}},R^{q_{2i}},h_{1i},h_{2i}\right),$
together with their corresponding abstractions $\hat\Sigma_i=(\R^{\hat n_i},\R^{\hat
m_i},\R^{\hat p_i},\hat{\mathcal{U}}_i,\hat{\mathcal{W}}_i,\hat f_i,\R^{q_{1i}},\R^{\hat q_{2i}},\hat
h_{1i},\hat h_{2i})$ and with storage functions $V_i$ from
$\hat\Sigma_i$ to $\Sigma_i$. We use $\alpha_{i}$, $\eta_i$, $\rho_{i\mathrm{ext}}$, $H_i$, $W_i$, $\hat W_i$, $X_i$, $X_i^{11}$, $X_i^{12}$, $X_i^{21}$, and $X_i^{22}$ to denote the corresponding functions, matrices, and their corresponding conformal block partitions appearing in Definition \ref{d:stf}.

The next theorem, one of the main results of the paper, provides a compositional approach on the construction of abstractions of networks of control systems and that of the corresponding simulation functions.

\begin{theorem}\label{t:ic}
Consider the interconnected control system
$\Sigma=\mathcal{I}(\Sigma_1,\ldots,\Sigma_N)$ induced by $N\in\N_{\geq1}$ control subsystems~$\Sigma_i$ and the coupling matrix $M$.
Suppose each control subsystem $\Sigma_i$ admits an abstraction $\hat \Sigma_i$ with the corresponding storage function $V_i$. If there exist $\mu_{i}\geq0$, $i\in[1;N]$, and matrix $\hat M$ of appropriate dimension such that the matrix (in)equality
\begin{align}\label{e:SGcondition}
\begin{bmatrix}
WM\\
I_{\tilde q}
\end{bmatrix}^T X(\mu_1X_1,\ldots,\mu_NX_N)\begin{bmatrix}
WM\\
I_{\tilde q}
\end{bmatrix}&\preceq0,\\\label{e:SGcondition1}
WMH&=\hat W\hat M,
\end{align}
are satisfied, where $\tilde q=\sum_{i=1}^Nq_{2i}$ and
\begin{align}\label{matrix1}
&W:=\mathsf{diag}(W_1,\ldots,W_N),~\hat W:=\mathsf{diag}(\hat W_1,\ldots,\hat W_N),~H:=\mathsf{diag}(H_1,\ldots,H_N),\\\label{matrix}&X(\mu_1X_1,\ldots,\mu_NX_N):=\begin{bmatrix}
\mu_1X_1^{11}&&&\mu_1X_1^{12}&&\\
&\ddots&&&\ddots&\\
&&\mu_NX_N^{11}&&&\mu_NX_N^{12}\\
\mu_1X_1^{21}&&&\mu_1X_1^{22}&&\\
&\ddots&&&\ddots&\\
&&\mu_NX_N^{21}&&&\mu_NX_N^{22}
\end{bmatrix},
\end{align}
then 
\begin{IEEEeqnarray*}{c}
V(x,\hat x)\Let\sum_{i=1}^N\mu_iV_i(x_i,\hat x_i)
\end{IEEEeqnarray*}
is a simulation
function from the interconnected control system $\hat \Sigma=\mathcal{I}(\hat
\Sigma_1,\ldots,\hat\Sigma_N)$, with the coupling matrix $\hat M$, to $\Sigma$. 
\end{theorem}

\begin{proof}
First we show that inequality \eqref{e:sf:10} holds for some
$\mathcal{K}_\infty$ function $\alpha$. For any
$x=\intcc{x_1;\ldots;x_N}\in\R^{n}$ and 
$\hat x=\intcc{\hat x_1;\ldots;\hat x_N}\in\R^{\hat n}$, one gets:
\begin{align*}
\Vert h(x)-\hat h(\hat x) \Vert&=\Vert [h_{11}(x_1);\ldots;h_{1N}(x_N)]-[\hat h_{11}(\hat x_1);\ldots;\hat h_{1N}(\hat x_N)] \Vert\\\notag
&\le\sum_{i=1}^N \Vert  h_{1i}(\hat x_i)-\hat h_{1i}(x_i) \Vert
\le \sum_{i=1}^N \alpha_{i}^{-1}(V_i( x_i, \hat x_i))\le \ol\alpha(V(x,\hat x)),
\end{align*}where
$\overline\alpha$ is a $\mathcal{K}_\infty$ function defined as 
\begin{align}\notag
	\ol\alpha(s) &\Let \left \{ \begin{array}{cc} \max\limits_{\as {\ge 0}} & \sum_{i=1}^N \alpha_{i}^{-1}(s_i) \\ \st & \mu^T \as=s, \end{array} \right. 
\end{align}
where $\as=\intcc{s_1;\ldots;s_N}\in\R^N$ and $\mu=\intcc{\mu_1;\ldots;\mu_N}$. By defining the
$\mathcal{K}_\infty$ function $\alpha(s)=\overline\alpha^{-1}(s)$, $\forall s\in\R_{\ge0}$, one obtains
$$\alpha(\Vert h(x)-\hat h(\hat x)\Vert)\le V( x, \hat x),$$satisfying inequality \eqref{e:sf:10}.
Now we show that inequality \eqref{inequality10} holds as well.
Consider any 
$x=\intcc{x_1;\ldots;x_N}\in\R^{n}$,
$\hat x=\intcc{\hat x_1;\ldots;\hat x_N}\in\R^{\hat n}$, and
$\hat u=\intcc{\hat u_{1};\ldots;\hat u_{N}}\in\R^{\hat m}$. For any $i\in[1;N]$, there exists $u_i\in\R^{m_i}$, consequently, a vector $u=\intcc{u_{1};\ldots;u_{N}}\in\R^{m}$, satisfying~\eqref{inequality1} for each pair of subsystems $\Sigma_i$ and $\hat\Sigma_i$
with the internal inputs given by $\intcc{w_1;\ldots;w_N}=M[h_{21}(x_1);\ldots;h_{2N}(x_N)]$ and $\intcc{\hat w_1;\ldots;\hat w_N}=\hat M[\hat h_{21}(\hat x_1);\ldots;\hat h_{2N}(\hat x_N)]$.
We derive the following inequality  
\begin{align}\label{rewrite}
&\dot V\left(x, \hat{x}\right)=\sum_{i=1}^N\mu_i\dot V_i\left( x_i,
\hat{x}_i\right)\\\notag&\leq\sum_{i=1}^N\mu_i\bigg(-\eta_i(V_i( x_i,\hat x_i))+\rho_{i\mathrm{ext}}(\Vert \hat u_i\Vert)+\begin{bmatrix}
W_iw_i-\hat W_i\hat w_i\\
h_{2i}(x_i)-H_i\hat h_{2i}(\hat x_i)
\end{bmatrix}^T\begin{bmatrix}
X_i^{11}&X_i^{12}\\
X_i^{21}&X_i^{22}
\end{bmatrix}\begin{bmatrix}
W_iw_i-\hat W_i\hat w_i\\
h_{2i}(x_i)-H_i\hat h_{2i}(\hat x_i)
\end{bmatrix}\bigg).
\end{align}
Using conditions \eqref{e:SGcondition} and \eqref{e:SGcondition1} and the definition of matrices $W$, $\hat W$, $H$, and $X$ in \eqref{matrix1} and \eqref{matrix}, the inequality \eqref{rewrite} can be rewritten as
\begin{align}\notag
\dot V\left(x, \hat{x}\right)
%\leq&\sum_{i=1}^N-\mu_i\eta_i(V_i( x_i,\hat x_i))+\sum_{i=1}^N\mu_i\rho_{i\mathrm{ext}}(\Vert \hat u_i\Vert)\\\notag&+\begin{bmatrix}
%W_1w_1-\hat W_1\hat w_1\\
%\vdots\\
%W_Nw_N-\hat W_N\hat w_N\\
%h_{21}(x_1)-H_1\hat h_{21}(\hat x_1)\\
%\vdots\\
%h_{2N}(x_N)-H_N\hat h_{2N}(\hat x_N)
%\end{bmatrix}^TX(\mu_1X_1,\ldots,\mu_NX_N)\begin{bmatrix}
%W_1w_1-\hat W_1\hat w_1\\
%\vdots\\
%W_Nw_N-\hat W_N\hat w_N\\
%h_{21}(x_1)-H_1\hat h_{21}(\hat x_1)\\
%\vdots\\
%h_{2N}(x_N)-H_N\hat h_{2N}(\hat x_N)
%\end{bmatrix}\\\notag
\leq&\sum_{i=1}^N-\mu_i\eta_i(V_i( x_i,\hat x_i))+\sum_{i=1}^N\mu_i\rho_{i\mathrm{ext}}(\Vert \hat u_i\Vert)\\\notag&+\begin{bmatrix}
W\begin{bmatrix}
w_1\\
\vdots\\
w_N
\end{bmatrix}-\hat W\begin{bmatrix}
\hat w_1\\
\vdots\\
\hat w_N
\end{bmatrix}\\
h_{21}(x_1)-H_1\hat h_{21}(\hat x_1)\\
\vdots\\
h_{2N}(x_N)-H_N\hat h_{2N}(\hat x_N)
\end{bmatrix}^TX(\mu_1X_1,\ldots,\mu_NX_N)\begin{bmatrix}
W\begin{bmatrix}
w_1\\
\vdots\\
w_N
\end{bmatrix}-\hat W\begin{bmatrix}
\hat w_1\\
\vdots\\
\hat w_N
\end{bmatrix}\\
h_{21}(x_1)-H_1\hat h_{21}(\hat x_1)\\
\vdots\\
h_{2N}(x_N)-H_N\hat h_{2N}(\hat x_N)
\end{bmatrix}\\\notag
\leq&\sum_{i=1}^N-\mu_i\eta_i(V_i( x_i,\hat x_i))+\sum_{i=1}^N\mu_i\rho_{i\mathrm{ext}}(\Vert \hat u_i\Vert)\\\notag&+\begin{bmatrix}
h_{21}(x_1)-H_1\hat h_{21}(\hat x_1)\\
\vdots\\
h_{2N}(x_N)-H_N\hat h_{2N}(\hat x_N)
\end{bmatrix}^T\begin{bmatrix}
WM\\
I_{\tilde q}
\end{bmatrix}^TX(\mu_1X_1,\ldots,\mu_NX_N)\begin{bmatrix}
WM\\
I_{\tilde q}
\end{bmatrix}\begin{bmatrix}
h_{21}(x_1)-H_1\hat h_{21}(\hat x_1)\\
\vdots\\
h_{2N}(x_N)-H_N\hat h_{2N}(\hat x_N)
\end{bmatrix}\\\notag
\leq&\sum_{i=1}^N-\mu_i\eta_i(V_i( x_i,\hat x_i))+\sum_{i=1}^N\mu_i\rho_{i\mathrm{ext}}(\Vert \hat u_i\Vert).
\end{align}

Define the functions 
\begin{subequations}
\begin{align}
	\label{lambda}
	\eta(s) &\Let \left \{ \begin{array}{cc} \min\limits_{\as  {\ge 0}} & \sum_{i=1}^N\mu_i\eta_i(s_i) \\ \st & \mu^T \as=s, \end{array} \right. \\
	\label{rho}
	\rho_{\mathrm{ext}}(s) &\Let \left \{ \begin{array}{cc} \max\limits_{\as {\ge 0}} & \sum_{i=1}^N\mu_i\rho_{i\mathrm{ext}}(s_i) \\ \st & \|\as \| = s, \end{array} \right. 
\end{align}
\end{subequations}
where $\eta\in\mathcal{K}_\infty$ and $\rho_{\mathrm{ext}}\in\mathcal{K}_\infty\cup\{0\}$.
By construction, we readily have
\begin{align}\notag
\dot V\left( x,\hat{x}\right)\leq-\eta\left(V\left( x,\hat{x}\right)\right)+\rho_{\mathrm{ext}}(\left\Vert \hat{u}\right\Vert),
\end{align}
which satisfies inequality \eqref{inequality10}. Hence, we conclude that $V$ is a simulation function from $\hat \Sigma$ to $\Sigma$. 
%In the derivation
%of \eqref{chain}, we can interchange
%the role of $\hat u$ and $u$ and see that for every
%$x\in\R^n$, $\hat x\in\R^{\hat n}$, $u\in\R^{\hat m}$ there
%exists $\hat u\in\R^m$ so that~\eqref{chain} holds. It follows that
%$V$ is a \MZ{SBF-M$_k$ function between $\hat \Sigma$ and $\Sigma$}.
%which completes the proof since it is clear that $\eta,\rho_{\mathrm{ext}}\in\mathcal{K}_\infty$.
\end{proof}

Figure \ref{composition1} illustrates schematically the result of
Theorem~\ref{t:ic}. 

\begin{remark}
Let us assume, $\forall i\in[1;N]$, $W_i=I_{p_i}$ and each control subsystem $\Sigma_i$ is single-internal-input single-internal-output. Under these assumptions, analytical feasibility conditions for matrix inequality \eqref{e:SGcondition} can be derived for special interconnection matrices $M$ including negative and positive feedback interconnection, skew symmetric interconnection, negative feedback cyclic interconnection, and finally extension to cactus graphs as provided in details in \cite[Chapter 2]{murat16}.
\end{remark}

\begin{figure}[t]
\begin{center}
\includegraphics[width=8.8cm]{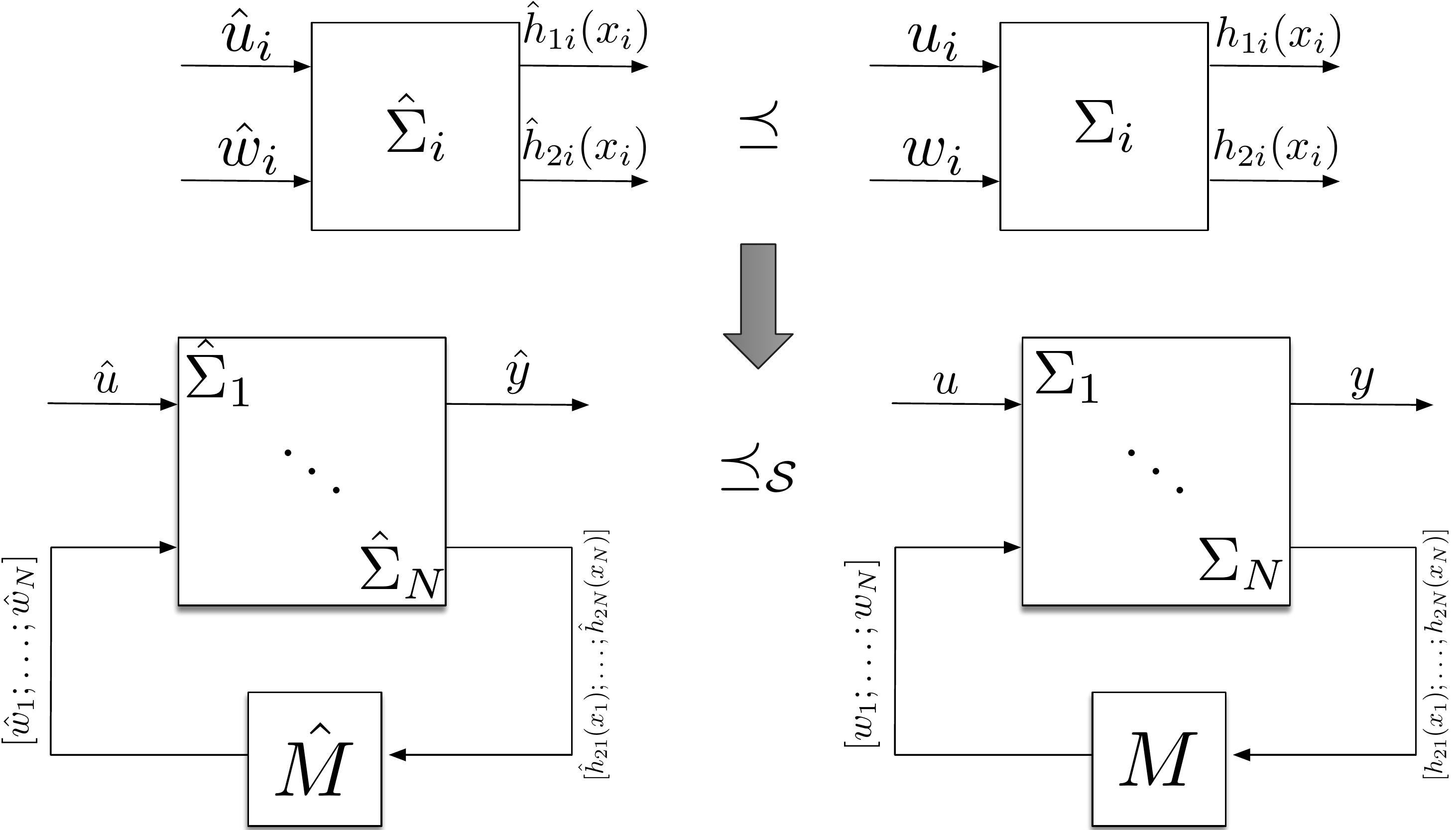}
\caption{Compositionality results provided that conditions \eqref{e:SGcondition} and \eqref{e:SGcondition1} are satisfied.}
\label{composition1}
\end{center}
\end{figure}

\section{Abstraction Synthesis for a Class of Nonlinear Control Systems}
\label{s:lin}

In this section, we concentrate on a specific class of nonlinear control systems $\Sigma$ and
\emph{quadratic} storage functions $V$. In the first part, we formally define the specific class of nonlinear control systems with which we deal in this section.
In the second part, we assume that an abstraction $\hat
\Sigma$ is given and we provide conditions under which $V$ is a storage
function. In the third part it is shown geometrically how to construct the
abstraction $\hat \Sigma$ together with the storage function $V$. Finally, we discuss the feasibility of a key condition based on which the results of this section hold.

\subsection{A class of nonlinear control systems}
The class of nonlinear control systems, considered in this section, is given by
\begin{IEEEeqnarray}{c}\label{e:lin:sys}\Sigma:\left\{
  \begin{IEEEeqnarraybox}[\relax][c]{rCl}
    \dot \xi&=&A\xi+E\varphi(F\xi)+B\upsilon+D\omega,\\
\zeta_1&=&C_1\xi,\\
\zeta_2&=&C_2\xi,
  \end{IEEEeqnarraybox}\right.
\end{IEEEeqnarray}
where $\varphi:\R\rightarrow\R$ satisfies 
\begin{equation}\label{rate}
a\leq\frac{\varphi(v)-\varphi(w)}{v-w}\leq b,~~~\forall v,w\in\R,v\neq w,
\end{equation}
for some $a\in\R$ and $b\in\R_{>0}\cup\{\infty\}$, $a\leq b$, and
\begin{IEEEeqnarray}{c,c,c,c,c,c,c,c,c}\nonumber
A\in\R^{n\times n},&
E\in\R^{n\times 1},&
F\in\R^{1\times n},&
B\in\R^{n\times m},&
D\in\R^{n\times p},&
C_1\in\R^{q_1\times n},&
C_2\in\R^{q_2\times n}.
\end{IEEEeqnarray}

We use the tuple
\begin{IEEEeqnarray*}{c}
\Sigma=(A,B,C_1,C_2,D,E,F,\varphi),
\end{IEEEeqnarray*}
to refer to the class of control systems of
the form~\eqref{e:lin:sys}.

\begin{remark}
If $\varphi$ in \eqref{e:lin:sys} is linear including the zero function (i.e. $\varphi\equiv0$) or $E$ is a zero matrix, one can remove or push the term $E\varphi(F\xi)$ to $A\xi$ and, hence, the tuple representing the class of control systems reduces to the linear one $\Sigma=(A,B,C_1,C_2,D)$. Therefore, every time we use the tuple $\Sigma=(A,B,C_1,C_2,D,E,F,\varphi)$, it implicitly implies that $\varphi$ is nonlinear and $E$ is nonzero. 
\end{remark}

Similar to what is shown in \cite{murat1}, without loss of generality, we can assume $a=0$ in \eqref{rate} for the class of nonlinear control systems in \eqref{e:lin:sys}. If $a\neq0$, one can define a new function $\widetilde\varphi(r):=\varphi(r)-ar$ which satisfies \eqref{rate} with $\widetilde a=0$ and $\widetilde b=b-a$, and rewrite \eqref{e:lin:sys} as
\begin{IEEEeqnarray}{c}\notag\Sigma:\left\{
  \begin{IEEEeqnarraybox}[\relax][c]{rCl}
    \dot \xi&=&\widetilde A\xi+E\widetilde\varphi(F\xi)+B\upsilon+D\omega,\\
\zeta_1&=&C_1\xi,\\
\zeta_2&=&C_2\xi,
  \end{IEEEeqnarraybox}\right.
\end{IEEEeqnarray}
where $\widetilde A=A+aEF$.

\begin{remark}\label{nonlinear}
For simplicity of derivations, we restrict ourselves to systems with a single nonlinearity as in \eqref{e:lin:sys}. However, it would be straightforward to obtain analogous results for systems with multiple nonlinearities as
\begin{IEEEeqnarray}{c}\notag\Sigma:\left\{
  \begin{IEEEeqnarraybox}[\relax][c]{rCl}
    \dot \xi&=&A\xi+\sum_{i=1}^ME_i\varphi_i(F_i\xi)+B\upsilon+D\omega,\\
\zeta_1&=&C_1\xi,\\
\zeta_2&=&C_2\xi
  \end{IEEEeqnarraybox}\right.
\end{IEEEeqnarray}
where $\varphi_i:\R\rightarrow\R$ satisfies \eqref{rate} for some $a_i\in\R$ and $b_i\in\R_{>0}\cup\{\infty\}$, $E_i\in\R^{n\times 1}$, and $F_i\in\R^{1\times n}$, for any $i\in[1;M]$. Furthermore, the proposed results here can also be extended to systems with multivariable nonlinearities satisfying  a multivariable sector property along the same lines as in \cite{fan} in the context of observer design.
\end{remark}

Note that the class of nonlinear control systems in \eqref{e:lin:sys} and Remark \ref{nonlinear} has been used widely to model many physical systems including active magnetic bearing \cite{murat1}, flexible joint robot \cite{fan}, fuel cell \cite{murat3}, the power generators \cite{scholtz}, underwater vehicles \cite{aamo}, and so on.

\subsection{Quadratic storage functions}
Here, we consider a quadratic storage function of the form 
\begin{align}\label{e:lin:sf}
V(x,\hat x)=(x-P\hat x)^T\widehat M(x-P\hat x),
\end{align}
where $P$ and $\widehat M\succ0$ are some matrices of appropriate
dimensions. In order to show that $V$ in \eqref{e:lin:sf} is a storage function from an abstraction $\hat\Sigma$ to a concrete system $\Sigma$, we require the following key assumption on $\Sigma$. 

\begin{assumption}\label{assumption1}
Let $\Sigma=(A,B,C_1,C_2,D,E,F,\varphi)$. Assume that for some constant
$\widehat\kappa\in\R_{>0}$ there exist matrices $\widehat M\succ0$, $K$, $L_1$, $Z$, $W$, $X^{11}$, $X^{12}$, $X^{21}$, and $X^{22}\preceq0$ of appropriate dimensions such that the matrix equality
\begin{align}\label{e:lin:con2f}
D&= ZW,
\end{align}
and inequality \eqref{e:lin:con11} hold, where $0$'s in \eqref{e:lin:con11} denote zero matrices of appropriate dimensions. 
\begin{figure*}
\begin{align}\label{e:lin:con11}
\begin{bmatrix}
\left(A+BK\right)^T\widehat M+\widehat M\left(A+BK\right) & \widehat MZ & \widehat M(BL_1+E)\\
Z^T\widehat M & 0 & 0\\
(BL_1+E)^T\widehat M& 0 & 0
\end{bmatrix}&\preceq\begin{bmatrix}
-\widehat\kappa\widehat M+C_2^TX^{22}C_2 & C_2^TX^{21}& -F^T\\
X^{12}C_2 & X^{11} & 0\\
-F& 0 & \frac{2}{b}
\end{bmatrix}
\end{align}
--------------------------------------------------------------------------------------------------------------------------------------------
\vspace{-5mm}
\end{figure*}
\end{assumption}

The next rather straightforward result provides a necessary and sufficient geometric condition for the existence of matrix $W$ appearing in condition \eqref{e:lin:con2f}.
\begin{lemma}\label{l:lin:con:W}
Given $D$ and $Z$, condition~\eqref{e:lin:con2f} is satisfied for some matrix $W$ if and only if
\begin{IEEEeqnarray}{c}\label{e:lin:con:W}
\im D\subseteq \im Z.
\end{IEEEeqnarray}
\end{lemma}
Note that the feasibility characterization of LMI \eqref{e:lin:con11} is more involved and will be discussed in details at the end of this section.

Note that matrix inequality \eqref{e:lin:con11} is bilinear in the variables $\widehat M$, $K$, $L_1$, $Z$, and linear in the variables $X^{11}$, $X^{12}$, $X^{21}$, and $X^{22}$ when we fix the constant $\widehat\kappa$. 
However, by assuming $C_2$ is a square and invertible matrix and introducing new variables $\ol K=K\widehat M^{-1}$, $\ol M=\widehat M^{-1}$, $\ol X^{22}=\widehat M^{-1}C_2^TX^{22}C_2\widehat M^{-1}$, $\ol X^{21}=\widehat M^{-1}C_2^TX^{21}$, $\ol X^{12}=X^{12}C_2\widehat M^{-1}$, and multiplying \eqref{e:lin:con11} from both sides by $$\begin{bmatrix}
\widehat M^{-1}&0&0\\0&I_n&0\\0&0&1
\end{bmatrix},$$where $0$'s denote zero matrices of appropriate dimension,
one obtains the matrix inequality \eqref{LMI} which is an LMI (linear matrix inequality) in the variables $\ol M$, $\ol K$, $L_1$, $Z$, $\ol X^{22}$, $\ol X^{21}$, $\ol X^{12}$, and $X^{11}$ when $\widehat \kappa$ is a fixed constant. 
\begin{figure*}
\begin{align}\label{LMI}
\begin{bmatrix}
\ol MA^T+A\ol M+\ol K^TB^T+B\ol K & Z & BL_1+E\\
Z^T & 0 & 0\\
E^T+L_1^TB^T& 0 & 0
\end{bmatrix}&\preceq\begin{bmatrix}
-\widehat\kappa \ol M+\ol X^{22} & \ol X^{21}& -\ol MF^T\\
\ol X^{12} & X^{11} & 0\\
-F\ol M& 0 & \frac{2}{b}
\end{bmatrix}
\end{align}
--------------------------------------------------------------------------------------------------------------------------------------------
%\vspace{-5mm}
\end{figure*}

\begin{remark}
Note that one can combine the compositionality condition \eqref{e:SGcondition} with a simultaneous search for quadratic storage functions \eqref{e:lin:sf} for subsystems of the form \eqref{e:lin:sys}. In particular, assume we are given $N$ control subsystems $\Sigma_i=(A_i,B_i,C_{1i},C_{2i},D_i,E_i,F_i,\varphi_i)$, $\forall i\in[1;N]$. For any $i\in[1;N]$, one can consider matrices $X_i$ in the LMI \eqref{e:SGcondition} as decision variables instead of being fixed and $\mu_i=1$ without loss of generality, and solve the combined feasibility problems \eqref{e:lin:con11} and \eqref{e:SGcondition}. Although the combined feasibility problem may be huge for large networks and solving it directly may be intractable, one can use the alternating direction method of multipliers (ADMM) to solve the feasibility problem in a distributed fashion along the same lines proposed in \cite{meissen}.
\end{remark}

Now, we provide one of the main results of this section showing under which conditions $V$ in \eqref{e:lin:sf} is a storage function.
\begin{theorem}\label{t:lin:suf}
Let $\Sigma=(A,B,C_1,C_2,D,E,F,\varphi)$
and $\hat \Sigma=(\hat A,\hat B,\hat C_1,\hat C_2,\hat D,\hat E,\hat F,\varphi)$ with $q_1=\hat q_1$. Suppose Assumption \ref{assumption1} holds and that there exist matrices $P$, $Q$, $H$, $L_2$, and $\hat W$ such that
\begin{IEEEeqnarray}{rCl}\label{e:lin:con2} \IEEEyesnumber
\IEEEyessubnumber\label{e:lin:con2a} AP&=&P\hat A-BQ\\
\IEEEyessubnumber\label{e:lin:con2b} C_1P&=&\hat C_1\\
\IEEEyessubnumber\label{e:lin:con2c1} X^{12}C_2P&=&X^{12}H\hat C_2\\
\IEEEyessubnumber\label{e:lin:con2c2} X^{22}C_2P&=&X^{22}H\hat C_2\\
\IEEEyessubnumber\label{e:lin:con2d}  FP&=&\hat F\\
\IEEEyessubnumber\label{e:lin:con2e}  E&=&P\hat E-B(L_1-L_2)\\
%\IEEEyessubnumber\label{e:lin:con2f}  D&=&ZW\\
\IEEEyessubnumber\label{e:lin:con2g}  P\hat D&=&Z\hat W,
\end{IEEEeqnarray}
hold. Then, function $V$ defined in~\eqref{e:lin:sf} is a storage function from~$\hat\Sigma$ to~$\Sigma$.
\end{theorem}

Before providing the proof, we point out that there always exist matrices $\{\hat A,\hat B,\hat C_1,\hat C_2,\hat D,\hat E,\hat F\}$ satisfying \eqref{e:lin:con2} if $P=I_n$ implying that $\hat n=n$. Naturally, it is better to have the simplest abstraction $\hat\Sigma$ and, therefore, one should seek a $P$ with $\hat n$ as small as possible. We elaborate on the construction of $P$ satisfying \eqref{e:lin:con2} in details in the next subsection.

\begin{proof}
%Note that $V$ is continuously differentiable.
%We show that for every
%$x\in\R^n$, $\hat x\in\R^{\hat n}$, $\hat u\in\R^{\hat m}$, $\hat w\in\R^{\hat p}$, there exists $u\in\R^m$ such that for all $w\in\R^p$, $V$ satisfies $\frac{\lambda_{\min}(M)}{\lambda_{\max}(C_1^TC_1)}\Vert C_1x-\hat C_1\hat x\Vert^2\le V(x,\hat
%x)$, $\frac{\lambda_{\min}(M)}{\lambda_{\max}(C_2^TC_2)}\Vert C_2x-H\hat C_2\hat x\Vert^2\le V(x,\hat
%x)$, and
%%\begin{small}
%\begin{align}\notag
%\frac{\partial V(x,\hat x)}{\partial x}&
%(Ax+E\varphi(Fx)+Bu+Dw)+
%\frac{\partial V(x,\hat x)}{\partial \hat x}
%(\hat A\hat x+\hat E\varphi(\hat F\hat x)+\hat B\hat u+\hat D\hat w)\\\label{e:t:lin:main:1}
%&\le-(\widehat\kappa-\pi) V(x,\hat x)+\frac{2\Vert\sqrt{M}(B\widetilde R-P\hat B)\Vert^2}{\pi}\Vert\hat u\Vert^2+\frac{2}{\pi}(Dw-P\hat D\hat w)^TM(Dw-P\hat D\hat w),
%\end{align}
%%\end{small}
%for any positive constant $\pi<\widehat\kappa$ and some matrix $\widetilde R$ of appropriate dimension all implying the satisfaction of conditions \eqref{e:sf:1}, \eqref{e:sf:3}, and inequality \eqref{inequality1}, respectively.
From~\eqref{e:lin:con2b} and for all $x\in\R^n$, $\hat
x\in\R^{\hat n}$, we have $\Vert C_1x-\hat C_1\hat x\Vert^2=(x-P\hat
x)^TC_1^TC_1(x-P\hat x)$. It can be readily verified that 
$\frac{\lambda_{\min}(\widehat M)}{\lambda_{\max}(C_1^TC_1)}\Vert C_1x-\hat C_1\hat x\Vert^2\le V(x,\hat x)$ holds for all $x\in\R^n$, $\hat
x\in\R^{\hat n}$ implying that inequality \eqref{e:sf:1} holds with $\alpha(r)=\frac{\lambda_{\min}(\widehat M)}{\lambda_{\max}(C_1^TC_1)}r^2$ for any $r\in\R_{\geq0}$. 
%Similarly, from ~\eqref{e:lin:con2c} and for all $x\in\R^n$, $\hat
%x\in\R^{\hat n}$, one gets
%$\frac{\lambda_{\min}(\widehat M)}{\lambda_{\max}(C_2^TC_2)}\Vert C_2x-S\hat C_2\hat x\Vert^2\le V(x,\hat x)$ which implies that inequality \eqref{e:sf:3} holds with $H=S$. 
We proceed with showing that the inequality~\eqref{inequality1} holds. 
Note that 
\begin{align}\label{partial}
\frac{\partial V(x,\hat x)}{\partial x}
= 
2(x-P\hat x)^T\widehat M,
\frac{\partial V(x,\hat x)}{\partial \hat x}
=
-2(x-P\hat x)^T\widehat MP.
\end{align}
Given any $x\in\R^n$, $\hat x\in\R^{\hat n}$, and $\hat u\in\R^{\hat m}$, we choose $u\in\R^{m}$ via the following linear \emph{interface} function:
\begin{align}\label{e:lin:int}
u=&k(x,\hat x,\hat u):=K(x-P\hat x)+Q\hat x+\widetilde R\hat u+L_1\varphi(Fx)-L_2\varphi(FP \hat x),
\end{align}
for some matrix $\widetilde R$ of appropriate dimension.

By using the equations~\eqref{e:lin:con2a}, \eqref{e:lin:con2d}, and \eqref{e:lin:con2e} and the definition of the interface function in~\eqref{e:lin:int}, we get
\begin{align*}
&Ax+E\varphi(Fx)+Bk(x,\hat x, \hat u)+Dw-P(\hat A\hat x+\hat E\varphi(\hat F\hat x)+\hat
B\hat u+\hat D\hat w)=\\&(A+BK)(x-P\hat x)+(Dw-P\hat D\hat w)+(B\widetilde R-P\hat B)\hat u+(E+BL_1)(\varphi(Fx)-\varphi(FP\hat x)).
\end{align*}
Using \eqref{partial}, \eqref{e:lin:con2f}, and \eqref{e:lin:con2g}, we obtain the following expression for $\dot V(x,\hat x)$:
%\begin{small}
\begin{align*}
\dot V(x,\hat x)=&2(x-P\hat x)^T\widehat M
\big[
(A+BK)(x-P\hat x)
+
(ZWw-Z\hat W\hat w)
\\\notag&+
(B\widetilde R-P\hat B)\hat u
+(E+BL_1)(\varphi(Fx)-\varphi(FP\hat x))\big].
\end{align*}
%\end{small}
From the slope restriction \eqref{rate}, one obtains
\begin{equation}\label{slope}
\varphi(Fx)-\varphi(FP\hat x)=\delta(Fx-FP\hat x)=\delta F(x-P\hat x),
\end{equation}
where $\delta$ is a constant and depending on $x$ and $\hat x$ takes values in the interval $[0,b]$. Using \eqref{slope}, the expression for $\dot V(x,\hat x)$ reduces to:
\begin{align*}%\label{derivative}
&\dot V(x,\hat x)=2(x-P\hat x)^T\widehat M
\big[
((A+BK)+\delta(E+BL_1)F)(x-P\hat x)
+
Z(Ww-\hat W\hat w)
+
(B\widetilde R-P\hat B)\hat u\big].
\end{align*}
Using Young's inequality \cite{Young225} as $$ab\leq \frac{\epsilon}{2}a^2+\frac{1}{2\epsilon}b^2,$$ for any $a,b\in\R$ and any $\epsilon>0$,
and with the help of Cauchy-Schwarz inequality, \eqref{e:lin:con11}, \eqref{e:lin:con2c1}, and \eqref{e:lin:con2c2}, one gets the following upper bound for $\dot V(x,\hat x)$:
\begin{align*}
\dot V(x&,\hat x)=2(x-P\hat x)^T\widehat M\big[
((A+BK)+\delta(E+BL_1)F)(x-P\hat x)
+
Z(Ww-\hat W\hat w)
+
(B\widetilde R-P\hat B)\hat u\big]\\=& 
\begin{bmatrix}x-P\hat x\\Ww-\hat W\hat w\\\delta F(x-P\hat x)\end{bmatrix}^T\begin{bmatrix}
\left(A+BK\right)^T\widehat M+\widehat M\left(A+BK\right)&\widehat MZ&\widehat M(BL_1+E)\\
Z^T\widehat M & 0 & 0\\
(BL_1+E)^T\widehat M& 0 & 0
\end{bmatrix}\begin{bmatrix}x-P\hat x\\Ww-\hat W\hat w\\\delta F(x-P\hat x)\end{bmatrix}\\\notag&+2(x-P\hat x)^T\widehat M
(B\widetilde R-P\hat B)\hat u\\\le\notag&
\begin{bmatrix}x-P\hat x\\Ww-\hat W\hat w\\\delta F(x-P\hat x)\end{bmatrix}^T\begin{bmatrix}
-\widehat\kappa\widehat M+C_2^TX^{22}C_2&C_2^TX^{21}&-F^T\\
X^{12}C_2&X^{11}&0\\
-F&0&\frac{2}{b}
\end{bmatrix}\begin{bmatrix}x-P\hat x\\Ww-\hat W\hat w\\\delta F(x-P\hat x)\end{bmatrix}+2(x-P\hat x)^T\widehat M(B\widetilde R-P\hat B)\hat u\\=&
-\widehat\kappa V(x,\hat x)-2\delta(1-\frac{\delta}{b})(x-P\hat x)^TF^TF(x-P\hat x)+\begin{bmatrix}Ww-\hat W\hat w\\C_2x-H\hat C_2\hat x\end{bmatrix}^T\begin{bmatrix}
X^{11}&X^{12}\\
X^{21}&X^{22}
\end{bmatrix}\begin{bmatrix}Ww-\hat W\hat w\\C_2x-H\hat C_2\hat x\end{bmatrix}\\&+2(x-P\hat x)^T\widehat M(B\widetilde R-P\hat B)\hat u\\\le& -(\widehat\kappa-\pi) V(x,\hat x)+\frac{\Vert\sqrt{\widehat M}(B\widetilde R-P\hat B)\Vert^2}{\pi}\Vert\hat u\Vert^2+\begin{bmatrix}Ww-\hat W\hat w\\C_2x-H\hat C_2\hat x\end{bmatrix}^T\begin{bmatrix}
X^{11}&X^{12}\\
X^{21}&X^{22}
\end{bmatrix}\begin{bmatrix}Ww-\hat W\hat w\\C_2x-H\hat C_2\hat x\end{bmatrix},
\end{align*}
for any positive constant $\pi<\widehat\kappa$.

Using this computed upper bound, the inequality \eqref{inequality1} is satisfied with the functions
$\eta\in\mathcal{K}_\infty$, $\rho_{\mathrm{ext}}\in\mathcal{K}_\infty\cup\{0\}$, and the matrix $X$,
as $\eta(s):=(\widehat\kappa-\pi) s$, $\rho_{\mathrm{ext}}(s):=\frac{\Vert\sqrt{\widehat M}(B\widetilde R-P\hat B)\Vert^2}{\pi} s^2$, $\forall s\in\R_{\ge0}$, and $X=\begin{bmatrix}
X^{11}&X^{12}\\
X^{21}&X^{22}
\end{bmatrix}$.
\end{proof}

The next result shows that conditions~\eqref{e:lin:con2a}-\eqref{e:lin:con2e}
are actually necessary for~\eqref{e:lin:sf} being a 
storage function from $\hat\Sigma$ to $\Sigma$ provided that the structure of the interface function is as in~\eqref{e:lin:int} for some matrices $K$, $Q$, $\widetilde R$, $L_1$, and $L_2$ of appropriate dimension.

\begin{theorem}\label{t:lin:nec}
Let $\Sigma=(A,B,C_1,C_2,D,E,F,\varphi)$
and $\hat \Sigma=(\hat A,\hat B,\hat C_1,\hat C_2,\hat
D,\hat E,\hat F,\varphi)$ with $q_1=\hat q_1$. Suppose that 
$V$ defined in~\eqref{e:lin:sf} is a storage function from~$\hat\Sigma$ to~$\Sigma$ with
the interface $k$ given in~\eqref{e:lin:int}. Then,
equations~\eqref{e:lin:con2a}-\eqref{e:lin:con2e} hold.
\end{theorem}

\begin{proof}
Since $V$ is a storage function from $\hat\Sigma$ to $\Sigma$, there exists a $\mathcal{K}_\infty$
function $\alpha$ such that $\Vert C_1x-\hat C_1\hat x\Vert\le \alpha^{-1}\left(V(x,\hat
x)\right)$. From~\eqref{e:lin:sf}, it follows that 
$\Vert C_1P\hat x-\hat C_1\hat x\Vert\le \alpha^{-1}(V(P\hat x,\hat x))=0$
holds for all $\hat x\in\R^{\hat n}$ which
implies~\eqref{e:lin:con2b}. 

Let us consider the
inputs $\hat \upsilon\equiv0$, $\omega\equiv0$, and $\hat
\omega\equiv0$. Since $X^{22}\preceq0$, inequality \eqref{inequality1} reduces to 
\begin{align}\label{vder}
&\dot V(x,\hat x)\leq-\eta(V(x,\hat x))+(h_2(x)-H\hat h_2(\hat x))^TX^{22}(h_2(x)-H\hat h_2(\hat x))\leq-\eta(V(x,\hat x)),
\end{align}
for any $x\in\R^n$ and $\hat x\in\R^{\hat n}$. Using the results in Lemma 4.4 in \cite{LSW96} or Lemma 3.6 in \cite{majid19}, inequality \eqref{vder} implies the existence of a $\mathcal{KL}$ function $\vartheta$ such that
\begin{align}\label{e:lin:ineq}
V(\xi(t),\hat\xi(t))\leq\vartheta(V(\xi(0),\hat\xi(0)),t),
\end{align}
holds, where $\hat \upsilon\equiv0$, $\omega\equiv0$, $\hat
\omega\equiv0$, and $\upsilon$ is given by the interface function
$k$ in~\eqref{e:lin:int}. Then,
for all
$\xi(0)=P\hat\xi(0)$, \mbox{$t\ge0$}, and using~\eqref{e:lin:ineq}, we obtain 
$V(\xi(t),\hat \xi(t))=0$. Since $M$ is positive definite, we have 
\begin{IEEEeqnarray*}{c't'c}
%\xi_{x,\nu,\omega}(t)=P \hat \xi_{\hat x,\hat \nu,\hat \omega}(t)& and &
%\dot \xi_{x,\nu,\omega}(t)=P\dot{ \hat \xi}_{\hat x,\hat \nu,\hat \omega}(t)
\xi(t)=P \hat \xi(t)& and &
\dot \xi(t)=P\dot{ \hat \xi}(t),
\end{IEEEeqnarray*}
from which we derive that 
$$AP\hat x+BQ\hat x+(E+B(L_1-L_2))\varphi(FP\hat x)=P\hat A\hat x+P\hat E\varphi(\hat F\hat x)$$
holds for all $\hat x\in\R^{\hat n}$ and, hence,~\eqref{e:lin:con2a}, \eqref{e:lin:con2d}, and \eqref{e:lin:con2e} follows.
It remains to show that~\eqref{e:lin:con2c1} and \eqref{e:lin:con2c2} hold. First assume $X^{22}\neq0$. Since $\dot V(P\hat x,\hat x)=V(P\hat x,\hat x)=0$ and using the first inequality in \eqref{vder}, one gets $$(C_2P\hat x-H\hat C_2\hat x)^TX^{22}(C_2P\hat x-H\hat C_2\hat x)\geq0,$$for any $\hat x\in\R^{\hat n}$. Since $X^{22}\preceq0$ and by assumption $X^{22}\neq0$, one obtains $X^{22}(C_2P-H\hat C_2)=0$ which implies \eqref{e:lin:con2c2}. Now, let us consider the
inputs $\hat \upsilon\equiv0$, $\omega\not\equiv0$, and $\hat \omega\not\equiv0$. Therefore, inequality \eqref{inequality1} reduces to 
\begin{align}\label{vder1}
\dot V(x,\hat x)\leq&-\eta(V(x,\hat x))+(Ww-\hat W\hat w)^TX^{11}(Ww-\hat W\hat w)+2(Ww-\hat W\hat w)^TX^{12}(C_2x-H\hat C_2\hat x),
\end{align}
for any $x\in\R^n$, $\hat x\in\R^{\hat n}$, $w\in\R^{p}$, and $\hat w\in\R^{\hat p}$. From \eqref{vder1} and by choosing $x=0_n$ and $\hat x=0_{\hat n}$, one can readily verify that $X^{11}\succeq0$. Then,
for all
$x=P\hat x$, we obtain
\begin{align*}
&(Ww-\hat W\hat w)^TX^{11}(Ww-\hat W\hat w)+2(Ww-\hat W\hat w)^TX^{12}(C_2P-H\hat C_2)\hat x\geq0,
\end{align*}
for any $w$, $\hat w$, and $\hat x$, which implies $X^{12}(C_2P-H\hat C_2)=0$ and, hence, \eqref{e:lin:con2c1} holds.
%With the same line
%of reasoning as above with $\hat \upsilon\equiv0$ and $\omega\equiv\hat \omega\equiv w$ with
%$w\in\R^p$ (instead of $\omega\equiv\hat \omega\equiv0$), we obtain the equality 
%$AP\hat x+BQ\hat x+(E+B(L_1-L_2))\varphi(FP\hat x)+Dw+BSw=P\hat A\hat x+P\hat E\varphi(\hat F\hat x)+P\hat Dw$. Now
%using~\eqref{e:lin:con2a}, \eqref{e:lin:con2d}, and \eqref{e:lin:con2e} we obtain $Dw=P\hat Dw-BSw$. Since this
%holds for all $w\in\R^p$, the condition~\eqref{e:lin:con2b} follows which completes the proof.
\end{proof}

\begin{remark}
Note that matrix $\widetilde R$ is a free design parameter in the interface function \eqref{e:lin:int}. Using the results in~\cite[Proposition 1]{GP09}, we choose $\widetilde R$ to minimize function $\rho_{\mathrm{ext}}$ for $V$ and, hence, reduce the upper bound in \eqref{inequality0} on the error between the output behaviors of $\Sigma$ and $\hat\Sigma$. The choice of $\widetilde R$ minimizing $\rho_{\mathrm{ext}}$ is given by
\begin{IEEEeqnarray}{c}\label{e:lin:con:tildeR}
\widetilde R=(B^T\widehat MB)^{-1} B^T\widehat M P\hat B.
\end{IEEEeqnarray}
\end{remark}

%\begin{remark}\label{hatD}
%Theorem \ref{t:lin:suf} does not impose any condition on matrix $\hat D$. However, the choice of $\hat D$ will influence the satisfaction of condition \eqref{e:SGcondition1} in order to provide the compositional result in theorem \ref{t:ic}. Assume we are given $N$ linear control systems $\Sigma_i=(A_i,B_i,C_{1i},C_{2i},D_i,E_i,F_i,\varphi_i)$ and their corresponding abstractions $\hat\Sigma_i=(\hat A_i,\hat B_i,\hat C_{1i},\hat C_{2i},\hat D_i,\hat E_i,\hat F_i,\hat \varphi_i)$, for any $i\in[1;N]$, satisfying the requirements in Theorem \ref{t:lin:suf}. Condition \eqref{e:SGcondition1} boils down to
%\begin{align}\label{e:SGcondition11}
%\begin{bmatrix}
%D_1&&\\
%&\ddots&\\
%&&D_N
%\end{bmatrix}M\begin{bmatrix}
%H_1&&\\
%&\ddots&\\
%&&H_N
%\end{bmatrix}=\begin{bmatrix}
%P_1\hat D_1&&\\
%&\ddots&\\
%&&P_N\hat D_N
%\end{bmatrix}\hat M.
%\end{align}
%Therefore, choosing $\hat D_i=I_{\hat n_i}$, $\forall i\in[1;N]$, maximize the feasible set of matrices $\hat M$ satisfying \eqref{e:SGcondition11}.
%\end{remark}

So far, we extracted various conditions on the
original system matrices $\{A,B,C_1,C_2,D,E,F\}$, the abstraction matrices $\{\hat A,\hat B,\hat C_1,\hat C_2,\hat D,\hat E,\hat F\}$, and the ones appearing in~\eqref{e:lin:sf} and~\eqref{e:lin:int}. Those conditions
ensure that $V$ in~\eqref{e:lin:sf} is
a storage function from $\hat\Sigma$ to $\Sigma$ with the corresponding interface function in \eqref{e:lin:int} refining any control signal designed for $\hat\Sigma$ to the one for $\Sigma$.
Apparently, those requirements do not enforce any condition on matrix $\hat B$. For example, one can select $\hat B=I_{\hat n}$ making the abstract system $\hat \Sigma$ fully actuated. On the other hand, one can ask not only for the existence of a storage function from $\hat \Sigma$
to $\Sigma$, but additionally require that all the controllable behaviors (in the absence of internal inputs) of the
concrete system $\Sigma$ are preserved over the abstraction $\hat \Sigma$. We refer the interested readers to \cite[Subsection 4.1]{GP09} and \cite[Section V]{pappas} for more details on what we mean by preservation of controllable behaviors. 
%Preservation of behaviors of $\Sigma$ over $\hat \Sigma$ ensures that nice properties like
%\emph{controllability} of $\Sigma$ are preserved on $\hat \Sigma$ implying that we do not disregard controllable behaviors of $\Sigma$ when working on $\hat \Sigma$.

The next theorem requires a condition on $\hat B$ in order to guarantee the preservation of controllable behaviors of $\Sigma$ over $\hat \Sigma$.

\begin{theorem}\label{t:B}
Let $\Sigma=(A,B,C_1,C_2,D,E,F,\varphi)$
and $\hat \Sigma=(\hat A,\hat B,\hat C_1,\hat C_2,\hat
D,\hat E,\hat F,\varphi)$ with $q_1=\hat q_1$. Suppose that there
exist matrices $P$, $Q$, $L_1$, and $L_2$ satisfying~\eqref{e:lin:con2a} and \eqref{e:lin:con2e}, and that matrix $\hat B$ is
given by
\begin{IEEEeqnarray}{c}\label{e:lin:con:B}
\hat B=[\hat PB\;\; \hat PAG],
\end{IEEEeqnarray}
where $\hat P$ and $G$ are assumed to satisfy 
\begin{IEEEeqnarray}{rCl}\label{e:lin:con3} \IEEEyesnumber
\IEEEyessubnumber\label{e:lin:con3a} C_1&=&\hat C_1\hat P\\
\IEEEyessubnumber\label{e:lin:con3b} I_n&=&P\hat P+GT\\
\IEEEyessubnumber\label{e:lin:con3c} I_{\hat n}&=& \hat PP\\
\IEEEyessubnumber\label{e:lin:con3d} F&=&\hat F\hat P,
\end{IEEEeqnarray}
for some matrix $T$. Then, for every trajectory
  $(\xi,\zeta_1,\zeta_2,\upsilon,0)$~of~$\Sigma$ there exists a trajectory
  $(\hat \xi,\hat \zeta_1,\hat\zeta_2,\hat \upsilon,0)$ of $\hat \Sigma$ where $\hat\xi=\hat P\xi$ and $\zeta_1=\hat \zeta_1$ hold.% for all~$t\in\R_{\ge0}$.
\end{theorem}

\begin{proof}
Let  $(\xi,\zeta_1,\zeta_2,\upsilon,0)$ be a trajectory of~$\Sigma$. We are
going to show that $(\hat P \xi,\zeta_1,\hat\zeta_2,\hat \upsilon,0)$ with 
\begin{IEEEeqnarray*}{c;t;c}
\hat \upsilon
=
\begin{bmatrix}
\upsilon-Q\hat P\xi-(L_1-L_2)\varphi(F\xi)\\
T\xi
\end{bmatrix},
\end{IEEEeqnarray*}
is a trajectory of $\hat \Sigma$. We use \eqref{e:lin:con3b} and derive
\begin{align*}
\hat P \dot\xi=&\hat PA\xi+\hat PE\varphi(F\xi)+\hat PB\upsilon=
\hat PAP\hat P\xi+\hat PA(I_n-P\hat P)\xi+\hat PE\varphi(F\xi)+\hat PB\upsilon\\
=&
\hat PAP\hat P\xi+\hat PAGT\xi+\hat PE\varphi(F\xi)+\hat PB\upsilon.
\end{align*}
Now we use the equations~\eqref{e:lin:con2a}, \eqref{e:lin:con2e}, \eqref{e:lin:con3c}, and \eqref{e:lin:con3d} and the definition of $\hat B$ and $\hat \upsilon$ to derive
\begin{align*}
\hat P \dot\xi
=&
\hat P(P \hat A-BQ)\hat P\xi+\hat PAGT\xi+\hat P(P\hat E-B(L_1-L_2))\varphi(F\xi)+\hat PB\upsilon\\
=&
\hat A\hat P \xi-\hat PBQ\hat P\xi+\hat PAGT\xi+\hat E\varphi(\hat F\hat P\xi)-\hat PB(L_1-L_2)\varphi(F\xi)+\hat PB\upsilon\\
=&
\hat A\hat P \xi+\hat E\varphi(\hat F\hat P\xi)+[\hat PB\; \hat PAG]\hat\upsilon=\hat A\hat P \xi+\hat E\varphi(\hat F\hat P\xi)+\hat B\hat\upsilon,
\end{align*}
showing that $(\hat P \xi,\hat \zeta_1,\hat\zeta_2,\hat \upsilon,0)$ is a trajectory of
$\hat \Sigma$. From $C_1=\hat C_1\hat P$ in \eqref{e:lin:con3a}, it follows that $\hat \zeta_1=\zeta_1$
which concludes the proof.
\end{proof}

\begin{remark}
Note that the previous result establishes that $\hat\Sigma$ (in the absence of internal inputs) is $\hat P$-related to $\Sigma$ as in \cite[Definition 3]{GP09}. We refer the interested readers to \cite{pappas} for more details about properties (e.g. controllability) of $\Phi$-related systems for some surjective smooth map $\Phi$.
\end{remark}

\subsection{Construction of abstractions}
Here, we provide several straightforward sufficient and necessary geometric conditions on matrices appearing in the definition of $\hat\Sigma$, of storage function and its corresponding interface function. The proposed geometric conditions facilitate the constructions of such matrices. First, we recall \cite[Lemma 2]{GP09} providing necessary and sufficient conditions for the existence of matrices $\hat A$ and $Q$ appearing in condition \eqref{e:lin:con2a}.

\begin{lemma}\label{l:lin:con:AQ}
Consider matrices $A$, $B$, and $P$. There exist matrices $\hat A$ and
$Q$ satisfying~\eqref{e:lin:con2a} if and only if
\begin{IEEEeqnarray}{c}\label{e:lin:con:AQ}
\im AP\subseteq \im P+\im B.
\end{IEEEeqnarray}
\end{lemma}

%Now, we recall the result proposed in \cite[Lemma 5.9]{majid19} providing necessary and sufficient conditions for the existence of matrices $\hat D$ and $S$ satisfying condition \eqref{e:lin:con2b}.
%
%\begin{lemma}\label{l:lin:con:DS}
%Given $P$ and $B$, there exist matrices $\hat D$ and
%$S$ satisfying~\eqref{e:lin:con2b} if and only if
%\begin{IEEEeqnarray}{c}\label{e:lin:con:DS}
%\im D\subseteq \im P+\im B.
%\end{IEEEeqnarray}
%\end{lemma}

Now, we give necessary and sufficient conditions for the
existence of matrices $\hat C_2$, $\hat E$, and $L_2$ appearing in
conditions~\eqref{e:lin:con2c1}, \eqref{e:lin:con2c2}, and \eqref{e:lin:con2e}, respectively.

\begin{lemma}\label{l:lin:con:H}
Given $P$, $C_2$, and $X^{12}$ (resp. $X^{22}$), there exists matrix $\hat C_2$ satisfying~\eqref{e:lin:con2c1} (resp. \eqref{e:lin:con2c2}) if and only if
\begin{align}\label{e:lin:con:H}
\im X^{12}C_2P\subseteq \im X^{12}H,~~(\text{resp.}~\im X^{22}C_2P\subseteq \im X^{22}H)
\end{align}
for some matrix $H$ of appropriate dimension.
\end{lemma}

\begin{lemma}\label{l:lin:con:EL}
Given $P$, $B$, and $L_1$, there exist matrices $\hat E$ and $L_2$ satisfying~\eqref{e:lin:con2e} if and only if
\begin{IEEEeqnarray}{c}\label{e:lin:con:EL}
\im E\subseteq \im P+\im B.
\end{IEEEeqnarray}
\end{lemma}

Lemmas~\ref{l:lin:con:AQ},~\ref{l:lin:con:H},~and \ref{l:lin:con:EL} provide
necessary and sufficient conditions on $P$ and $H$ resulting in the
construction of matrices $\hat A$, $\hat C_2$, and $\hat E$ together with the matrices
$Q$ and $L_2$ appearing in the definition of the interface function in \eqref{e:lin:int}. Matrices $\hat F$ and $\hat C_1$ are computed as $\hat F=FP$ and $\hat C_1=C_1P$. The next lemma provides a necessary and sufficient condition on the existence of matrix $\hat D$ appearing in condition \eqref{e:lin:con2g}.

\begin{lemma}\label{l:lin:con:hatD}
Given $Z$, there exists matrix $\hat D$ satisfying~\eqref{e:lin:con2g} if and only if
\begin{IEEEeqnarray}{c}\label{e:lin:con:hatD}
\im Z\hat W\subseteq\im P,
\end{IEEEeqnarray}
for some matrix $\hat W$ of appropriate dimension.
\end{lemma}
Although condition \eqref{e:lin:con:hatD} is readily satisfied by choosing $\hat W=0$, one should preferably aim at finding a nonzero $\hat W$ to smooth later the satisfaction of compositionality condition \eqref{e:SGcondition1}.

As we already mentioned, the choice of matrix $\hat B$ is free.  One can also construct $\hat B$ as in~\eqref{e:lin:con:B} ensuring preservation of all controllable behaviors of $\Sigma$ over $\hat\Sigma$ under extra conditions given in \eqref{e:lin:con3}. Lemma 3 in \cite{GP09}, as recalled next, provides necessary and
sufficient conditions on $P$ and $C_1$ for the existence of $\hat P$,
$G$, and $T$ satisfying \eqref{e:lin:con3a}, \eqref{e:lin:con3b}, and \eqref{e:lin:con3c}.

\begin{lemma}\label{l:lin:con:CP}
Consider matrices $C_1$ and $P$ with $P$ being injective and let $\hat C_1=C_1P$. There exists
matrix $\hat P$ satisfying~\eqref{e:lin:con3a}, \eqref{e:lin:con3b}, and \eqref{e:lin:con3c}, for some matrices $G$ and $T$ of appropriate dimensions, if and only
if
\begin{IEEEeqnarray}{c}\label{e:lin:con:BB}
\im P+\ke C_1=\R^n.
\end{IEEEeqnarray}
\end{lemma}

Similar to Lemma~\ref{l:lin:con:CP}, we give necessary and sufficient conditions on $P$ and $F$ for the existence of $\hat P$ satisfying \eqref{e:lin:con3d}.

\begin{lemma}\label{l:lin:con:FP}
Consider matrices $F$ and $P$ with $P$ being injective and let $\hat F=FP$. There exists
matrix $\hat P$ satisfying \eqref{e:lin:con3d} if and only
if
\begin{IEEEeqnarray}{c}\label{e:lin:con:BF}
\im P+\ke F=\R^n.
\end{IEEEeqnarray}
\end{lemma}

Note that conditions \eqref{e:lin:con11}, \eqref{e:lin:con:W}, and \eqref{e:lin:con:AQ}-\eqref{e:lin:con:hatD} (resp. \eqref{e:lin:con11}, \eqref{e:lin:con:W}, and \eqref{e:lin:con:AQ}-\eqref{e:lin:con:BF}) complete the characterization of mainly matrices $P$ and $Z$ which together with the matrices $\{A,B,C_1,C_2,D,E,F\}$ result in the construction of matrices
$\{\hat A,\hat B,\hat C_1,\hat C_2,\hat D,\hat E,\hat F\}$, where $\hat B$ can be chosen freely with appropriate dimensions (resp. $\hat B$ is computed as in \eqref{e:lin:con:B}). 
%We point out that there always exists an injective matrix $P\in\R^{n\times \hat n}$ satisfying conditions~\eqref{e:lin:con:AQ}-\eqref{e:lin:con:hatD} and~\eqref{e:lin:con:BB}-\eqref{e:lin:con:BF}. In the worst-case scenario, one can choose $P=I_n$ implying that $\hat n=n$. Naturally, it is better to have the simplest abstraction $\hat\Sigma$ and, therefore, one should seek a $P$ with $\hat n$ as small as possible.

We summarize the construction of the abstraction $\hat \Sigma$, storage function $V$ in \eqref{e:lin:sf}, and its corresponding interface function in \eqref{e:lin:int} in Table~\ref{tb:1}.

\begin{table}[h]
\caption{Construction of $\hat\Sigma=(\hat A,\hat B,\hat C_1,\hat C_2,\hat D,\hat E,\hat F,\varphi)$, the corresponding storage function $V$ in \eqref{e:lin:sf}, and interface function in \eqref{e:lin:int} for a given $\Sigma=(A,B,C_1,C_2,D,E,F,\varphi)$.}\label{tb:1}
\centering
\begin{tabular}{|p{0.0\columnwidth}l|}
\hline
1.&Compute matrices $\widehat M$, $K$, $L_1$, $Z$, $X^{11}$, $X^{12}$, and $X^{22}$ satisfying \eqref{e:lin:con11} and \eqref{e:lin:con:W}.\\
2.&Pick an injective $P$ with the lowest rank satisfying \eqref{e:lin:con:AQ}-\eqref{e:lin:con:hatD} (resp. \eqref{e:lin:con:AQ}-\eqref{e:lin:con:BF});\\
3.&Compute $\hat A$ and $Q$ from~\eqref{e:lin:con2a};\\
4.&Compute $\hat E$ and $L_2$ from~\eqref{e:lin:con2e};\\
5.&Compute $\hat F=FP$;\\
6.&Compute $\hat C_1=CP$;\\
6.&Compute $\hat C_2$ satisfying $H\hat C_2=CP$ for some $H$;\\
7.&Compute $\hat D$ satisfying $P\hat D=Z\hat W$ for some (rather nonzero) $\hat W$;\\
8.&Choose $\hat B$ freely (resp. $\hat B=[\hat PB\;\; \hat PAG]$);\\
9.&Compute $\widetilde R$, appearing in \eqref{e:lin:int}, from \eqref{e:lin:con:tildeR};\\
\hline
\end{tabular}
\end{table}

\subsection{Feasibility of LMI \eqref{e:lin:con11}}
In this subsection we discuss sufficient and necessary feasibility conditions for the LMI \eqref{e:lin:con11} in the restrictive case of $X^{12}=0$ and $X^{11}\succeq \frac{Z^T\widehat MZ}{\pi}$ for any positive constant $\pi<\widehat \kappa$, where $0$ denotes a zero matrix of appropriate dimension. To do so, we convert the feasibility conditions for the restricted version of LMI \eqref{e:lin:con11} into the ones for two dual control problems. When $b=\infty$ in \eqref{e:lin:con11}, the feasibility of restricted \eqref{e:lin:con11} is dual to the one of designing a controller rendering a linear system strictly positive real (SPR) \cite{murat2}. When $b<\infty$, the duality is with a linear $\mathcal{L}_2$-gain assignment control problem \cite[Section 13.2]{isidori}.

When $b=\infty$, the restricted version of LMI \eqref{e:lin:con11} reduces to 
\begin{align}\label{e:lin:con21}
\left(A+BK\right)^T\widehat M+\widehat M\left(A+BK\right)&\prec 0, \\\label{e:lin:con31}
\widehat M(BL_1+E)+F^T&=0.
\end{align}

By virtue of the Positive-Real Lemma \cite{vladimir}, conditions \eqref{e:lin:con21} and \eqref{e:lin:con31} mean that the linear control system
\begin{IEEEeqnarray}{c}\label{e:lin:sys2}\Sigma:\left\{
  \begin{IEEEeqnarraybox}[\relax][c]{rCl}
   \dot \xi&=&A\xi+B\upsilon+E\omega,\\
\zeta&=&-F\xi,%
  \end{IEEEeqnarraybox}\right.
\end{IEEEeqnarray}
is enforced SPR from the disturbance $\omega$ to the output $\zeta$ by the control law 
\begin{align}\label{SPR}
\upsilon=K\xi+L_1\omega.
\end{align}

Therefore, when $b=\infty$, the feasibility of the restricted version of LMI \eqref{e:lin:con11} is dual to the feasibility of the control problem in which the system \eqref{e:lin:sys2} is enforced SPR by the control law \eqref{SPR}.

When $b<\infty$, using the Schur complement of $-2/b$, one can readily verify that the restricted version of LMI \eqref{e:lin:con11} is equivalent to
\begin{align*}
\left(A+BK+\frac{b}{2}(BL_1+E)F\right)^T\widehat M&+\widehat M\left(A+BK+\frac{b}{2}(BL_1+E)F\right)\\\notag&+\frac{b}{2}\widehat M(BL_1+E)(BL_1+E)^T\widehat M+\frac{b}{2}F^TF\prec0,
\end{align*}
which means that the $\mathcal{L}_2$-gain of the dual system \eqref{e:lin:sys2} from input $\widetilde\omega:=\omega+(b/2)\zeta$ to output $\zeta$ is enforced to be strictly less than $2/b$ by the control law $\upsilon=K\xi+L_1\omega$ \cite[Section 13.2]{isidori}.

%In order to characterize the feasibility of those dual problems, the dual system \eqref{e:lin:sys2} is represented using a change of coordinate and a preliminary state feedback (always existing) as 
%\begin{align}\notag
%\dot\eta&=\Pi\eta+\Gamma\theta_1+\Lambda\omega\\\notag
%\dot\theta_1&=\theta_2+\lambda_1\omega\\\notag
%\dot\theta_2&=\theta_3+\lambda_2\omega\\\label{change}
%&\vdots\\\notag
%\dot\theta_r&=\upsilon+
%\end{align}

We refer the interested readers to \cite[Theorem 3]{murat1} deriving sufficient and necessary feasibility conditions for the restricted version of LMI \eqref{e:lin:con11} by looking into the corresponding dual control problems, namely, enforcing SPR and assigning a linear $\mathcal{L}_2$-gain. 

Note that in the context of observer design and observer-based control, the feasibility of those dual control problems have been investigated for several physical problems in \cite{murat1,fan,murat3,scholtz}.

\section{Example}\label{case}
Consider a linear control system $\Sigma=(-L,I_n,C)$ satisfying 
\begin{IEEEeqnarray*}{c}
 \Sigma:\left\{
  \begin{IEEEeqnarraybox}[\relax][c]{rCl}
    \dot\xi&=&-L\xi+\upsilon,\\
    \zeta&=&C\xi,%
  \end{IEEEeqnarraybox}\right.
\end{IEEEeqnarray*}
for some matrix $C\in\R^{q\times n}$ and $L\in\R^{n\times n}$. Assume $L$ is the Laplacian matrix \cite{godsil} of an undirected graph, {\it e.g.}, for a complete graph:
\begin{equation}\label{complete}
L=\begin{bmatrix}n-1 & -1 & \cdots & \cdots & -1 \\  -1 & n-1 & -1 & \cdots & -1 \\ -1 & -1 & n-1 & \cdots & -1 \\ \vdots &  & \ddots & \ddots & \vdots \\ -1 & \cdots & \cdots & -1 & n-1\end{bmatrix},
\end{equation}
and $C$ has the following block diagonal structure$$C=\mathsf{diag}(C_{11},\ldots,C_{1N}),$$where $C_{1i}\in\R^{q_{1i}\times n_i}$. We partition $\xi$ as $\xi=[\xi_1;\ldots;\xi_N]$ and $\upsilon$ as $\upsilon=[\upsilon_1;\ldots;\upsilon_N]$ where $\xi_i$ and $\upsilon_i$ are both taking values in $\mathbb{R}^{n_i}$, $\forall i\in[1;N]$. Now, by introducing $\Sigma_i=(0_{n_i},I_{n_i},C_{1i},I_{n_i},I_{n_i})$ satisfying
\begin{IEEEeqnarray*}{c}
 \Sigma_i:\left\{
  \begin{IEEEeqnarraybox}[\relax][c]{rCl}
    \dot\xi_i&=&\omega_i+\upsilon_i,\\
    \zeta_{1i}&=&C_{1i}\xi_i,\\
    \zeta_{2i}&=&\xi_i,%
  \end{IEEEeqnarraybox}\right.
\end{IEEEeqnarray*}
one can readily verify that $\Sigma=\mathcal{I}(\Sigma_1,\ldots,\Sigma_N)$ where the coupling matrix $M$ is given by $M=-L$.

Our goal is to aggregate each $\xi_i$ taking values in $\mathbb{R}^{n_i}$ into a scalar-valued $\hat{\xi}_i$, governed by $\hat\Sigma_i=(0,1,C_{1i}\mathbf{1}_{n_i},1,1)$ which satisfies:
\begin{IEEEeqnarray*}{c}
 \hat\Sigma_i:\left\{
  \begin{IEEEeqnarraybox}[\relax][c]{rCl}
    \dot{\hat\xi}_i&=&\hat\omega_i+\hat\upsilon_i,\\
    \hat\zeta_{1i}&=&C_{1i}\mathbf{1}_{n_i}\hat\xi_i,\\
    \hat\zeta_{2i}&=&\hat\xi_i.%
  \end{IEEEeqnarraybox}\right.
\end{IEEEeqnarray*}

One can readily verify that, for any $i\in[1;N]$, conditions \eqref{e:lin:con2f} and \eqref{e:lin:con11} are satisfied with $\widehat M_i=I_{n_i}$, $K_i=-\lambda I_{n_i}$, for some $\lambda>0$, $\widehat\kappa_i=2\lambda$, $Z_i=I_{n_i}$, $L_{1i}=0$, $W_i=I_{n_i}$, $X^{11}=0$, $X^{22}=0$, and $X^{12}=X^{21}=I_{n_i}$, where $0$ denotes zero matrices of appropriate dimensions. Moreover, for any $i\in[1;N]$, $P_i=\mathbf{1}_{n_i}$ satisfies conditions \eqref{e:lin:con2} with $Q_i={L_2}_i=0_{n_i}$, $H=\mathbf{1}_{n_i}$, and $\hat W_i=\mathbf{1}_{n_i}$. Hence, function $V_i(x_i,\hat{x}_i)=(x_i-\mathbf{1}_{n_i}\hat x_i)^T(x_i-\mathbf{1}_{n_i}\hat x_i)$ is a storage function from $\hat\Sigma_i$ to $\Sigma_i$ satisfying condition \eqref{e:sf:1} with $\alpha_{i}(r)=\frac{1}{\lambda_{\max}(C_{1i}^TC_{1i})}r^2$ and condition \eqref{inequality1} with $\eta(r)=-2\lambda r$, $\rho_{\text{ext}}(r)=0$, $\forall r\in\R_{\ge0}$, $W_i=I_{n_i}$, $\hat W_i=H_i=\mathbf{1}_{n_i}$, and 
\begin{equation}\label{passivity}
X_i=\begin{bmatrix} 0 & I_{n_i} \\ I_{n_i} &  0 \end{bmatrix},
\end{equation}
where the input $u_i\in\R^{n_i}$ is given via the interface function in \eqref{e:lin:int} as $u_i=-\lambda(x_i-\mathbf{1}_{n_i}\hat x_i)+\mathbf{1}_{n_i}\hat u_i$. Note that $\widetilde R_i=\mathbf{1}_{n_i}$ was computed as in \eqref{e:lin:con:tildeR}.

Now, we look at $\hat\Sigma=\mathcal{I}(\hat\Sigma_1,\ldots,\hat\Sigma_N)$ with a coupling matrix $\hat M$ satisfying condition \eqref{e:SGcondition1} as follows:
\begin{equation}\label{graph}
-L\mathsf{diag}(\mathbf{1}_{n_1},\ldots,\mathbf{1}_{n_N})=\mathsf{diag}(\mathbf{1}_{n_1},\ldots,\mathbf{1}_{n_N})\hat M.
\end{equation}

Note that the existence of $\hat M$ satisfying (\ref{graph}) for a graph Laplacian $L$ means that the $N$ subgraphs form an {\it equitable partition} of the full graph \cite{godsil}. Although this restricts the choice of a partition in general, for the complete graph (\ref{complete}) any partition is equitable.

Choosing $\mu_1=\cdots=\mu_N=1$ and using $X_i$ in (\ref{passivity}), matrix $X$ in \eqref{matrix} reduces to
$$
X=\begin{bmatrix} 0 & I_{n} \\ I_{n} &  0 \end{bmatrix},
$$
and condition \eqref{e:SGcondition} reduces to
$$
\begin{bmatrix} -L \\ I_n \end{bmatrix}^TX\begin{bmatrix} -L \\ I_n \end{bmatrix}=-L-L^T\preceq 0
$$
which always holds without any restrictions on the size of the graph. In order to show the above inequality, we used $L=L^T\succeq0$ which is always true for Laplacian matrices of undirected graphs. 

For the sake of simulation, we fix $n=9$ and $$C=\begin{bmatrix}1& 0& 0& 0& 0& 0& 0& 0& 0\\0 &0& 0& 0& 1& 0& 0& 0& 0\\0& 0& 0& 0& 0& 0& 0& 0& 1\end{bmatrix},$$where 
\begin{align*}
C_{11}=\begin{bmatrix}1& 0& 0\end{bmatrix},C_{12}=\begin{bmatrix}0& 1& 0\end{bmatrix},C_{13}=\begin{bmatrix}0& 0& 1\end{bmatrix}.
\end{align*}
Let us now synthesize a controller for $\Sigma$ via the abstraction $\hat\Sigma$ to enforce the specification, defined
by the LTL formula \cite{katoen08}
\begin{equation}
\label{LTL}
\psi = \square S\wedge\Big( \underset{i=1}{\overset{5}{\bigwedge}} \square(\lnot O_i)\Big) \wedge  \square\Diamond T_1 \wedge \square\Diamond T_2, 
\end{equation}
which requires that any output trajectory $\zeta$ of the closed loop
system evolves inside the set $S$, avoids sets $O_i$, $i\in[1;5]$, indicated with blue boxes in Figure \ref{fig1}, and visits each $T_i$, $i\in[1;2]$, indicated with red boxed in Figure \ref{fig1}, infinitely often. We use \texttt{SCOTS} \cite{scots} to
synthesize a controller for $\hat\Sigma$ to enforce \eqref{LTL}. In the synthesis
process we restricted the abstract inputs to $\hat u_1,\hat u_2,\hat u_3\in[-14,14]$. Given that we can set the initial
states of $\Sigma$ to $x_i = P_i\hat x_i$, so that $V_i (x_i, \hat x_i) = 0$, and since $\rho_{\text{ext}}(r)=0$, $\forall r\in\R_{\ge0}$, we obtain $\Vert \zeta(t)-\hat\zeta(t) \Vert=0$ for all $t\geq0$. A closed-loop output trajectory of $\Sigma$ is illustrated in Figure \ref{fig1}. Note that it would not have been possible to synthesize
a controller using \texttt{SCOTS} for the original 9-dimensional system $\Sigma$, without the 
3-dimensional intermediate approximation $\hat\Sigma$.

\begin{figure}
%\hspace{-0.5cm}
\begin{center}
\includegraphics[width=10cm]{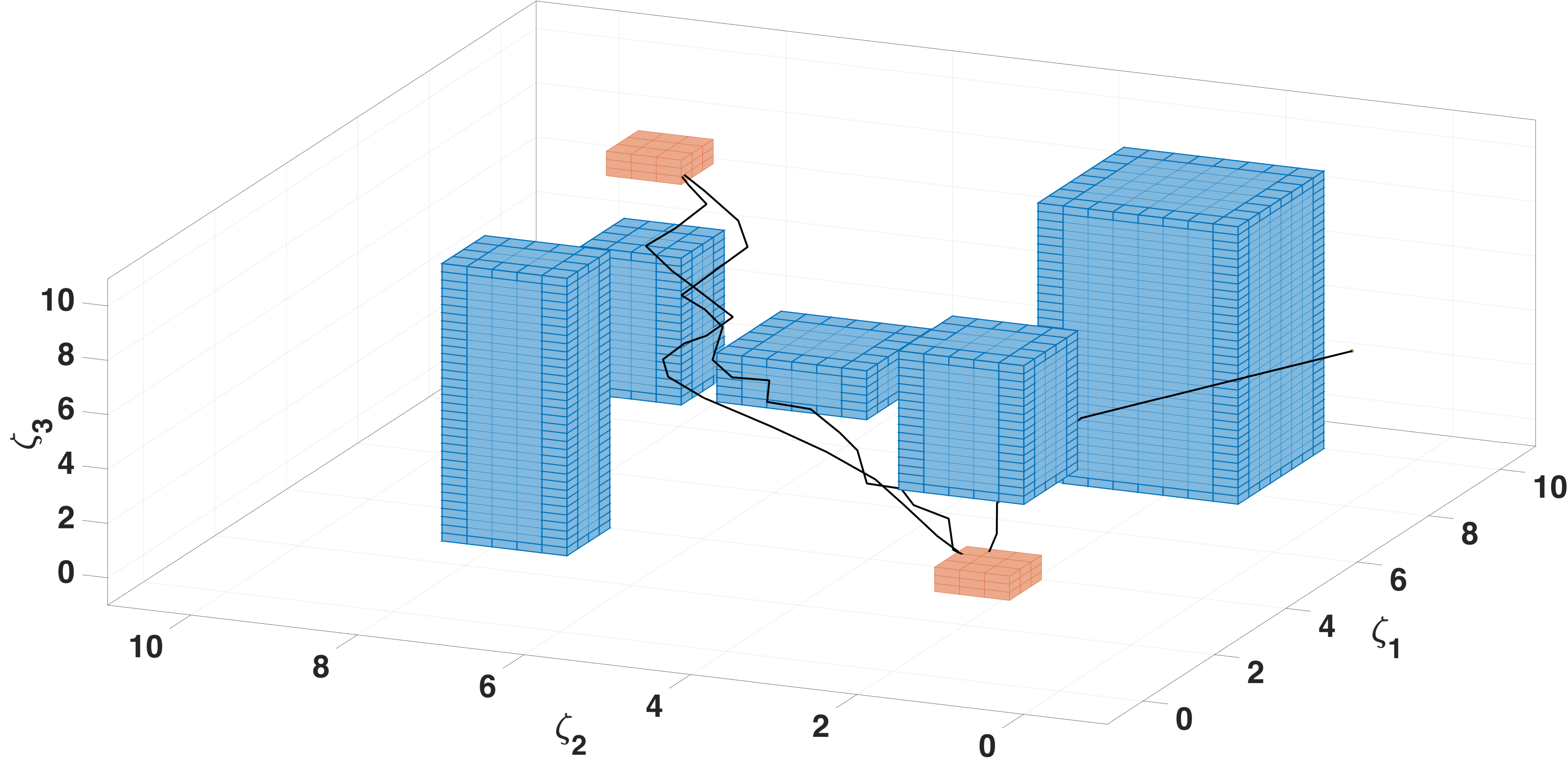}
\caption{The specification with closed loop output trajectory of $\Sigma$. The sets $S$, $O_i$, $i\in[1;5]$, and $T_i$, $i\in[1;2]$ are given by: $S=[0,10]^3$, $T_1=[1,2]^3$ and $T_2=[8,9]^3$, $O_1=[4,6]^3$, $O_2=[7,9]\times[1, 3]\times[0,10]$, $O_3=[2,3]\times[7, 8]\times[0,10]$, $O_4=[1,2]\times[1,2]\times[5,10]$, and $O_5=[8,9]\times[8,9]\times[0,5]$.}
\end{center}
\label{fig1}
\end{figure}

\begin{remark}
This scale-free result highlights the advantage of dissipativity-type over small-gain type conditions proposed in \cite{majid17,majid20}: the storage function $V_i$ from $\hat\Sigma_i$ to $\Sigma_i$ in this example also satisfies the requirements of a simulation function defined in \cite{majid17,majid20}; however, the resulting small-gain type condition, {\it e.g.}, for $L$ in (\ref{complete}) reduces to $\frac{n-1}{n-1+\lambda}<1$ which involves the spectral radius\footnote{The spectral radius of a square matrix $A\in\R^{n\times n}$, denoted by $\rho(A)$, is defined as $\rho(A):=\max\{\vert\lambda_1\vert,\cdots,\vert\lambda_n\vert\}$ where $\lambda_1,\ldots,\lambda_n$ are eigenvalues of $A$.} of $L$ ($\rho(L)=n$). Hence, using the results in \cite{majid17,majid20}, one can readily verify that as the number of components increases, {\it e.g.} $n\to\infty$, the quality of approximation deteriorates unless the interface gain $\lambda$ is increasing with $n$ which is not desirable because it results in high amplitude inputs $u_i$.
\end{remark}

\section{Conclusion}
In this paper, we proposed for the first time a notion of so-called storage function relating a concrete control system to its abstraction by quantifying their joint input-output correlation. This notion was adapted from the one of storage function from dissipativity theory. Given a network of control subsystems together with their corresponding abstractions and storage functions, we provide compositional conditions under which a network of abstractions approximate the original network and the approximation error can be quantified compositionally using the storage functions of the subsystems. Finally, we provide a procedure for the construction of abstractions together with their corresponding storage functions for a class of nonlinear control systems by using the bounds on the slope of system nonlinearities. One of the main advantages of the proposed results here based on a dissipativity-type condition in comparison with the existing ones based on a small-gain type condition is that the former can enjoy specific interconnection matrix and provide scale-free compositional conditions (cf. Section \ref{case}).

\section{Acknowledgments}
The authors would like to thank Mahmoud Khaled for the simulation in Section \ref{case}.

\bibliographystyle{alpha}
\bibliography{reference}

\end{document}